\documentclass{amsart}
\usepackage{amssymb}
\usepackage{amsmath}
\usepackage{color}
\def\R{{\mathbb R}}

\def\N{{\mathbb N}}

\def\tt{\mathbf{t}}
\def\ii{\mathbf{i}}
\def\0{\mathbf{0}}
\def\tb{\tilde b}

\newtheorem{theo}{Theorem}
\newtheorem{prop}{\indent Proposition}
\newtheorem{lem}{\indent Lemma}
\newtheorem{defin}{\indent Definition}
\newtheorem{ex}{\indent Example}
\newtheorem{rem}{\indent Remark}
\newtheorem{ass}[theo]{Assumption}
\newtheorem{cor}{\indent Corollary}

\begin{document}

\title[Invariant density of PDMP's]{Absolute continuity of the invariant measure in Piecewise Deterministic Markov Processes having degenerate jumps}

\author{E. L\"ocherbach }

\address{CNRS UMR 8088, D\'epartement de Math\'ematiques, Universit\'e de Cergy-Pontoise,
2 avenue Adolphe Chauvin, 95302 Cergy-Pontoise Cedex, France.}

\email{eva.loecherbach@u-cergy.fr}



\date{January 25, 2016}

\begin{abstract}
We consider piecewise deterministic Markov processes with degenerate transition kernels of the {\it house-of-cards-}type. We use a splitting scheme based on jump times to prove the absolute continuity, as well as some regularity, of the invariant measure of the process. Finally, we obtain finer results on the regularity of the one-dimensional marginals of the invariant measure, using integration by parts with respect to the jump times.  
\end{abstract}
\maketitle

\section{Introduction}
We consider an interacting particle system $X_t = (X_t^1, \ldots, X_t^N) $ 
taking values in $\R^N$ and solving, for $t\geq 0$,
\begin{equation}\label{eq:dyn}
X_t =  X_0 +   \int_0^t b ( X_s) ds +  \sum_{i=1}^N \int_0^t \int_0^\infty 
  a_i ( X_{s-} )    1_{ \{ z \le  f_i ( X_{s-}) \}} N^i (ds, dz).
\end{equation}  
Here, $N^i  (ds, dz), 1 \le i \le N,$ are independent Poisson random measures on $\R_+ \times \R_+ $ having intensity measure $ds dz $ each.  The function $b : \R^N  \to \R^N  $ is a smooth drift function, and for each $ 1 \le i \le N,$ $a_i : \R^N  \to \R^N  $ are jump functions, and $f_i :\R^N \mapsto \R_+ $ jump rate functions. The infinitesimal generator 
of the process $X$ is given for any smooth test function $ g : \R^N \to \R $ by
\begin{equation}\label{eq:generator0}
L g (x ) =  \sum_{i=1}^N  f_i(x) \left[ g (  x + a_i ( x)  ) - g (x) \right]
+  <  \nabla g (x),  b ( x) >   .
\end{equation}

We work under the  
assumption that there exists a unique non-exploding solution to \eqref{eq:dyn} which is recurrent in the sense of Harris having a unique invariant probability measure $m.$ 

The aim of the present paper is to study the smoothness of the invariant measure $m$  of this particle system in the case of degenerate transitions which are of the form 
\begin{equation}\label{eq:afterjump1}
 x + a_i ( x) = \left( \begin{array}{c}
x^1 + a_i^1 ( x)\\
\vdots \\
x^{i-1} + a_i^{  i-1} (x) \\
0 \\
x^{i+1} + a_i^{i+1} (x)\\
\vdots \\
x^N + a_i^{ N} (x) 
\end{array}
\right) 
\leftarrow \mbox{ coordinate } i , 
\end{equation}
for $ x = (x^1, \ldots , x^N ) \in \R^N.$ This means that a jump of particle $i $ leads to a reset of this particle's position to $0 $ and gives an additional $ a_i^j ( x)  $ to any other particle $j.$  We call such processes {\it house-of-cards-}like interacting particle systems. Systems of this type are good models for systems of interacting neurons as introduced by Galves and L\"ocherbach (2016) \cite{antonio-eva}, see also Duarte and Ost (2016) \cite{bresiliens} and Hodara et al. (2016) \cite{pierre}.

Notice that \eqref{eq:dyn} is a Piecewise Deterministic Markov process (PDMP) in the sense of Davis (1993) \cite{Davis93}. The process evolves according to the deterministic flow $ \gamma_{s, t } ( x) $ solution of 
$$ \gamma_{s, t } (x) = x + \int_s^t b( \gamma_{s, u} (x) ) du , s \le t ,$$
between successive jumps, and the only randomness is given by the random jump times and the choice of the (random) positions of the process right after the jump.  The jump rate of the process depends on the configuration of the process and is given by $ \bar f (x)= \sum_{i=1}^N f^i ( x) .$ The transition kernel $ Q ( x, dy ) = \sum_{i=1}^N  \frac{f_i ( x)}{\overline f (x) }  \delta_{ x + a_i ( x) } ( dy ) $ is (partly) degenerate if the jumps are governed by transitions as described in \eqref{eq:afterjump1}; indeed, in this case, not only transitions do not create density, but they even destroy density for the particles that jump -- which are reset to $ 0.$ As a consequence, we are in a very singular scheme here.

Invariant measures and densities of PDMP's or more generally of jump processes have been widely studied in the literature. An overwhelming number of articles is devoted to the study of the regularity of the transition semi-group, i.e.\ the study of the existence and regularity of a transition density.  For this purpose,  the Malliavin calculus for processes with jumps has been developed, using the regularity both created by the jump amplitudes or the jump times. We refer to the by now classical studies of Bichteler, Gravereaux  and Jacod (1987) \cite{bgj}, Bismut (1983) \cite{bismut}, Carlen and Pardoux (1990) \cite{carlen-pardoux},  Denis (2000) \cite{Denis} and Picard (1996) \cite{picard}. These papers deal with a much wider class of models including infinite jump activity and degenerate jump measures (in the sense that jumps do not create density). 

Concerning more specifically the world of PDMP's which are models having a finite jump activity depending on state space, only few results on  the regularity of the associated semi-group are available. Fournier (2002) \cite{fournier} exploits some monotonicity properties of the jumps -- but this monotonicity is not present in our situation, since transitions of the kind \eqref{eq:afterjump1} are inherently non-monotone.  
Concerning the invariant measure of PDMP's in an abstract frame,  we refer the reader to Costa and Dufour (2008) \cite{costa-dufour} for a general study of the stability properties of PDMP's.  Regarding the regularity of the invariant measure in PDMP models, in most cases this study  is based on the amplitude of the jumps, i.e.\ on some smoothness created by the transition kernel. This approach has been used e.g.\ by Biedrzycka and Tyran-Kaminska (2016) \cite{Tyran} in the case where the transition kernel transports Lebesgue absolute continuity.  An approach based on the jump times has been followed by Bena\"{i}m et al. (2015) \cite{micheletal}, but in the very specific situation of randomly switching systems of ODE's without jumps in the spatial variable. 

All these methods cannot be applied in our model, at least not directly.  The main reason for this is the fact that not only the jumps do not create smoothness, but that they destroy it partially. The fact that the transition kernel is not creating Lebesgue density implies that we have to use the noise present in the jump times. In this sense we are close to Carlen and Pardoux (1990) \cite{carlen-pardoux} or also to Bally and Cl\'ement (2010) \cite{BallyClementine}. 
Finally we have also been inspired by the approach proposed by Coquio and Gravereaux (1992) \cite{coquio}: They study the smoothness of the invariant measure of Markov chains, based on Malliavin calculus. In all these papers, {\it preservation of smoothness} is assumed in the sense that -- once created -- smoothness will not be destroyed again by the transition kernel. There is no such preservation of smoothness in our model.

To resume, we are facing a very singular situation where the only way of creating smoothness is by using the jump times and where transitions destroy accumulated smoothness partially. 

The present article gives conditions 
implying that the invariant measure of the process possesses a Lebesgue density. Moreover we study conditions under which this density is smooth.  In Section \ref{sec:2} we show how a {\it good} succession of jumps, characterized by a {\it good} order in which successive particles jump, can lead to the creation of Lebesgue density in $\R^N$ (compare to Definition \ref{def:good}). In a typical house-of-cards-like interacting particle system, a good order is given if first the first particle jumps, then the second, then the third, and so on. As a consequence, a {\it good succession of jumps} does not happen all the time and is a rather rare event. However, due to Harris recurrence, it is sufficient that this event has strictly positive probability. Then a regenerative argument (inspired by the well-known technique of Nummelin splitting) allows to deduce the following : There exists a stopping time $R  $ which is finite almost surely,  such that the position $X_R$ of the process at the stopping time possesses a smooth Lebesgue density  
(see Theorem \ref{theo:7} and Corollary \ref{cor:1}).  This result is achieved by using a change of variables based on the jump times. 

The most important point is then to study how this smoothness is transported by the dynamics. Already the next jump of, say, particle $i$ destroys the smoothness in direction of  $e_i, $ the $i-$th unit vector of $ \R^N.$ 
The main idea is to show that this destruction of density in direction of $e_i $ can be {\it counter-balanced} by the creation of density due to the next jump time. Loosely speaking, the exponential density of the next jumping time may create density in direction of $e_i $ -- under suitable conditions on the deterministic flow. The technical details are given in Theorem \ref{theo:8}, the main ingredient is a simple change of variables using the coarea formula.

The principal application we have in mind is given by systems of interacting neurons as in \cite{bresiliens}, \cite{antonio-eva} and \cite{pierre}. For these systems, it can be shown that our strategy works and that the system possesses an invariant probability measure which is absolutely continuous with respect to the Lebesgue measure. It is in \cite{bresiliens} that {\it good} successions of jumps have been introduced to create Lebesgue density in view of proving the Harris recurrence of the process.

The idea of using {\it favorable noise} which is eventually created by the system, coupled to the Harris recurrence and based on a regeneration approach, goes back to the work of  Poly (2012) \cite{poly}, compare to his Theorem 1.1 and his Remark 1.3. The main difference with his work is that he also imposes the preservation of smoothness property for the transition kernel. 
Using splitting methods to create noise has also been applied in 
Bally and Rey (2015) \cite{bally-rey}, however in a different context. 

Once we have obtained a Lebesgue density of the invariant measure, it is natural to ask for further smoothness properties. This is the content of Section \ref{sec:37} where we discuss the smoothness of the invariant density in systems where the only interactions between particles are given by the jumps. 
If we dispose of a good control of the balance between the explosion rate $ e^{Bt }$ of the inverse flow as time tends to infinity and the survival rate $ e^{ - \int_0^t \bar f(\gamma_s (x) ) ds } ,$ see Theorem \ref{theo:10}, then we obtain regularity of the invariant density up to some order which is given by the balance of these two rates. In particular, the invariant density will not be $ C^\infty $ in general, even if all coefficients of the system are supposed to be smooth.  

A second part of the paper is devoted to the study of the marginal density of a single particle in the invariant regime.  By using an integration by parts formula with respect to the jump times, we obtain  Theorem \ref{theo:invmeasure} which shows that the marginal density of a single particle in the invariant regime inherits the smoothness properties of the jump rate functions $f_i $ and of the drift function $b,$ locally on a set of positions which are far from the equilibria of the flow and far from $0  .$  

Our paper is organized as follows. In Section \ref{sec:2} we state our main assumptions, establish a useful relation with the invariant measure of the jump chain associated to the process -- the positions of the process just before jumping --  and give our regularity result concerning the invariant measure of a single particle in Theorem \ref{theo:invmeasure}. The proof of this theorem is given in the Appendix. Section \ref{sec:3} is devoted to the study of the invariant measure of the whole particle process. We start by introducing skeletons for the jump chain, study its derivatives with respect to the successive jump times, introduce the notion of a {\it good succession} of jumps in Definition \ref{def:good} and introduce the splitting procedure in Sections \ref{sec:34} and \ref{sec:35}. An application of the coarea formula allows to prove the Lebesgue absolute continuity of the invariant measure (Theorem \ref{theo:8} in Section \ref{sec:36}). Finally, in Section \ref{sec:37}, and particularly in Theorem \ref{theo:10}, we discuss the regularity properties of the invariant density.

\section{Main assumptions and regularity of marginals}\label{sec:2}

\subsection{The dynamics}

We consider $N$ independent Poisson random measures
$N^i  (ds, dz), 1 \le i \le N, $ on $\R_+ \times \R_+ $ having intensity measure $ds dz$ each and study the piecewise deterministic Markov process (PDMP)
$X_t = (X^{ 1 }_t, \ldots , X^{ N}_t )$
taking values in $\R^N$ and solving, for $t\geq 0$,
\begin{equation}\label{eq:dyn2}
X_t =  X_0 +   \int_0^t b ( X_s) ds +  \sum_{i=1}^N \int_0^t \int_0^\infty 
  a_i ( X_{s-} )    1_{ \{ z \le  f_i ( X_{s-}) \}} N^i (ds, dz).
\nonumber
\end{equation}  
The coefficients of this system are  the drift function $b : \R^N  \to \R^N  $ and for each $ 1 \le i \le N,$  jump functions $a_i : \R^N  \to \R^N  $  and jump rate functions $f_i :\R^N \mapsto \R_+, $ 
satisfying (at least) the following assumption.

\begin{ass}\label{ass:1}
1. $a_i : \R^N   \to \R^N , 1 \le i \le N ,$ are measurable.
  \\
2. $f_i, 1 \le i \le N ,$ are Lipschitz continuous such that  $ f_i (x) > 0 $ for all $x$ and all $i.$\\
3. $ b : \R^N \to \R^N  $ is Lipschitz continuous and of linear growth. 
\end{ass}

As a consequence of item 3.\ of the above assumption, we may introduce the 
deterministic flow $ \gamma_{s,t} (  v ) = ( \gamma_{s, t }^1 ( v ) , \ldots , \gamma_{s, t }^N ( v) ),$ solution of 
\begin{equation}\label{eq:flow}
\gamma_{s,t} (  x) =   x + \int_s^t b ( \gamma_{s, u } (x) ) du   , \; 0 \le s \le t , 
\end{equation}
for any starting configuration $x \in \R^N .$ This flow exists on $ 0 \le s \le t < \infty  ,$ due to the linear growth condition imposed on $ b.$ 

In most of the cases we will impose that $ \sup_i \sup_{x \in \R^N} f_i ( x) < \infty $ 
implying that there is no accumulation of jumps in finite time. As a consequence, there exists a unique non-exploding solution to \eqref{eq:dyn} for any starting configuration $X_0 = x.$ \footnote{Of course, a finer study of conditions ensuring the existence of a non-exploding solution to \eqref{eq:dyn} can be conducted, but this is outside the scope of the present paper.} 

We write
\begin{equation}
 \Delta_i ( x) := x + a_i ( x) ,
\end{equation}
for the configuration of the process after a jump of particle $i.$ Moreover, we introduce the short hand notation 
\begin{equation}
\bar f (x) = \sum_{i=1}^N f_i (x) 
\end{equation}
which is the total jump rate of the system, when it is in configuration $x.$ 
We suppose that 

\begin{ass}\label{ass:2}
For all $x \in \R^N, $ 
\begin{equation}\label{eq:finitejumptime}
   \int_0^\infty  \bar f ( \gamma_s (x) ) ds  = \infty ;
\end{equation}
\end{ass}
implying that the process will jump infinitely often almost surely. 

Let $ T_0 = 0 < T_1 < T_2 \ldots < T_n < \ldots $ be the successive jump times of the process, defined by 
$$ T_{n+1} = \inf \{ t > T_n : X_t \neq X_{t- } \} , n \geq 0 .$$
We introduce the jump measure 
$$ \mu ( ds, dy , dz  ) = \sum_{n \geq 1 } 1_{ \{  T_n < \infty \} }  \delta_{ (T_n, X_{T_n - } , X_{T_n} ) }  (dt, dy, dz).$$

By our assumptions, $ \mu $ is compensated by $ \nu ( ds, dy , dz) = \sum_{i=1}^N  f_i ( X_s) ds \delta_{ X_s} ( dy)  \delta_{  \Delta_i ( X_s) } (dz) .$ 

Finally, we impose the following condition.
\begin{ass}\label{ass:3}
The process $X$ is recurrent in the sense of Harris, with invariant probability measure $ m ;$ i.e.\ for any $ O \in {\mathcal B}( \R^N ) $ with $ m( O ) > 0 ,$ we have $P_x -$almost surely, $ \lim \sup_{t \to \infty } 1_O ( X_t) = 1 ,$ for any $x \in \R^N.$ Moreover, we suppose that $ \int f_i (x) m (dx) < \infty $ for all $ 1 \le i \le N,$ i.e.\ the total jump rate is integrable with respect to the invariant measure.
\end{ass}

\begin{rem}
The purpose of the present paper is not to establish recurrence conditions ensuring that Assumption \ref{ass:3} holds.  We refer the reader to Costa and Dufour (2008) \cite{costa-dufour} for a general treatment of the stability properties of PDMP's and to Duarte and Ost (2015) \cite{bresiliens} or to Hodara et al. (2016) \cite{pierre} for examples of processes that follow our model assumptions, which are systems of interacting neurons where the Harris recurrence has been proven.  
\end{rem} 

We will be mainly interested in the situation where the jumps are given by  
$$ a_i (x) = \left( 
\begin{array}{c}
*\\
\vdots \\
*\\
- x_i  \\
* \\
\vdots\\
* 
\end{array}
\right) \longleftarrow \mbox{ coordinate } i 
 . $$
In other words, a jump of the particle $i$ leads to a reset of particle $i$ to the position $0, $ and raises the positions of the other particles $j, j \neq i , $ to the new position $x^j + a_i (x)^j = \Delta_i^j ( x) .$ We call such processes {\it house-of-cards-}like interacting particle systems.
 
In this case, the transition kernel associated to the jumps of system \eqref{eq:dyn} 
$$ Q ( x, dy) = {\mathcal L} ( X_{T_1} |X_{T_1- } = x ) (dy ) =  \sum_{ i = 1 }^N \frac{f_i ( x) }{\bar f (x) } \delta_{ \Delta_i(x) } ( dy ) $$
is degenerate since the $i-$th component $ (x+ a_i (x))^i = \Delta_i^i ( x) =  0 $ for all $x .$ 


\subsection{An associated Markov chain and its invariant measure}

We start with some simple preliminary considerations. Let $Z_k = X_{T_k- } , k \geq 1 , $ be the jump chain. Then the following holds.

\begin{prop}
Grant Assumptions \ref{ass:1}--\ref{ass:3}.
$(Z_k)_k$ is Harris recurrent with invariant measure $m^Z$ given by 
$$ m^Z ( g) = \frac{1}{m ( \bar f ) } m ( \bar f g ) , $$
for any $g : \R^N \to \R$ measurable and bounded.
\end{prop}

\begin{proof}
Let $ g$ be a bounded test function. It is sufficient to prove that 
$ \frac{1}{n} \sum_{k=1}^n g ( Z_k) \to m^Z ( g) $
as $n \to \infty, $ $P_x-$almost surely, for any fixed starting point $x .$ But
$$ \frac{1}{n} \sum_{k=1}^n g ( Z_k) = \frac{1}{n} \sum_{k=1}^n g ( X_{T_k-}),$$
and, putting $ N_t = \sup \{ n : T_n \le t \},$  
$$
 \lim_{ n \to \infty }  \frac{1}{n} \sum_{k=1}^n g ( X_{T_k-}) = \lim_{t \to \infty}\frac{t}{N_t} \frac{1}{t} \sum_{k=1}^{N_t} g ( X_{T_k-}) 
= \lim_{t \to \infty}\frac{t}{N_t} \frac{1}{t} \int_0^t\int_{\R^N}\int_{\R^N} g (y) \mu ( ds, dy, dz) .
$$
By the ergodic theorem, $ N_t/t \to \int  \bar f(x) m (dx ) = m( \bar f ) ,$ and this convergence holds almost surely. Moreover, 
\begin{equation}\label{eq:martingaleplusbiaisconv}
 \frac{1}{t} \int_0^t\int_{\R^N} \int_{\R^N} g (y) \mu ( ds, dy, dz)  = \frac{1}{t} M_t + \frac{1}{t} \int_0^t\int_{\R^N} \int_{\R^N} g (y) \nu ( ds, dy, dz) ,
\end{equation}
where $ M_t = \int_0^t \int \int g ( y) [\mu (ds, dy, dz) - \nu  (ds, dy, dz) ].$ Then $M_t$ is in ${\mathcal M}^{2,d}_{loc} ,$ the set of all
locally square integrable purely discontinuous martingales,
with predictable quadratic covariation process
\begin{equation}
<M>_t\; =\int_0^t g^2 (X_s)  \bar f (X_s )ds 
\end{equation}
where
$$ \frac{< M>_t}{t} \to m ( g^2 \bar f ) $$
almost surely, as $t \to \infty.$ By the martingale convergence theorem, see e.g.\ Jacod-Shiryaev (2003) \cite{js}, $ t^{-1/2} M_t $ converges in law to a normal distribution. As a consequence, $M_t/t \to 0$ almost surely. 

We now treat the second term in \eqref{eq:martingaleplusbiaisconv}. By the ergodic theorem for integrable additive functionals, 
$$ \frac{1}{t} \int_0^t\int_{\R^N } \int_{\R^N}g (y) \nu ( ds, dy, dz) = \frac{1}{t} \int_0^t g (X_s) \bar f( X_s) ds \to m ( \bar f g ),$$
and this finishes the proof.
\end{proof}

As a consequence of the above proposition, to prove the absolute continuity of $m,$ it is sufficient to show that the invariant measure $m^Z $ of the chain $(Z_k)_k$ is absolutely continuous with respect to the Lebesgue measure. 

We first obtain a useful representation formula for $m$ from the above proposition. Introduce 
\begin{equation}
 e (x,t) = e^{ - \int_0^t  \bar f ( \gamma_s ( x) ) ds }  ,
\end{equation}
which is the survival rate of the process starting from position $x.$ 

\begin{prop}
Let $ g : \R^N \to \R$ be a bounded test function. Then
\begin{equation}\label{eq:useful}
m ( g) = \sum_{i=1}^N \int_{\R^N} m (dx) f_i (x)  \int_0^\infty e (\Delta_i ( x) , t) g ( \gamma_t (\Delta_i (x) ) ) dt.
\end{equation}
\end{prop}

\begin{proof}
We have by the relation between $m^Z$ and $m,$  
$$
m ( g) =  m ( \bar f ) E_{m^Z}  \left( g ( Z_n )  \frac{1}{ \bar f (Z_n) }\right) . 
$$
We use that $ m^Z = {\mathcal L}(X_{T_2-} | X_{T_1-} \sim m^Z ) . $ 
Then we obtain
\begin{eqnarray}
m (g) &=& m ( \bar f ) m^Z \left(  \frac{g ( Z_n  ) }{ \bar f (Z_n) }\right) \nonumber \\
&=&  \int_{\R^N}  \bar f (x) m (dx) \int Q ( x, dy)  \int_0^\infty \bar f ( \gamma_t (  y ))  e^{- \int_0^t \bar f ( \gamma_s (  y)) ds } g ( \gamma_t (  y ) ) \frac{1 }{ \bar f ( \gamma_t (  y ))} dt \nonumber \\
&=&  \sum_{i=1}^N \int_{\R^N} m ( dx)  f_i (x)   \int_0^\infty e( \Delta_i ( x) , t)  g( \gamma_t (\Delta_i ( x) ) ) dt . 
\end{eqnarray}
\end{proof}

As a corollary of the above representation, we deduce that the invariant measure of the process is absolutely continuous if $N=1,$ i.e.\ in the one-dimensional case.

\begin{cor}\label{theo:4}
Grant Assumptions \ref{ass:1}, \ref{ass:2} and \ref{ass:3}.
Suppose that $N = 1$ and moreover that $b$ is one times differentiable, having a bounded derivate. Then $m$ is absolutely continuous with respect to the Lebesgue measure on $ {\mathcal A} := \{ x : b ( x) \neq 0 \} ,$ having a continuous density on $ {\mathcal A}.$ 
\end{cor}

\begin{proof}
Suppose that $ {\mathcal A} \neq \emptyset $ (otherwise, we do not have to prove anything), and let $ g$ be a smooth test function having compact support included in $ \{ x : |b(x) | \geq \varepsilon \} \subset {\mathcal A},$ for some fixed $ \varepsilon > 0.$  We obtain
\begin{eqnarray}\label{eq:dim1}
m ( g'  ) &=& \int_{\R} m ( dx) f(x)  \int_0^\infty \frac{e( \Delta(x) , t)}{b ( \gamma_t ( \Delta(x)  ))} g'  ( \gamma_t ( \Delta (x) ) ) \frac{d \gamma_t ( \Delta (x)  )}{dt }  dt  \nonumber \\
&=&\int_{\R} m ( dx) f(x)   \int_0^\infty \frac{e (\Delta (x) , t)}{b  ( \gamma_t ( \Delta (x) ))} \frac{d}{dt} \left( g ( \gamma_t (\Delta (x)  ) ) \right) dt ,
\end{eqnarray}
where $ \Delta ( x) = x + a (x) .$ Integration by parts yields that 
\begin{multline*}
 \int_0^\infty \frac{e (\Delta(x) , t)}{b  ( \gamma_t ( \Delta(x) ))} \frac{d}{dt} \left( g ( \gamma_t (\Delta(x) ) ) \right) dt   \\
= [ \frac{e (\Delta(x) ,t)}{b ( \gamma_t (\Delta(x) ) ) }  g ( \gamma_t (\Delta(x) ) ]_0^\infty - \int_0^\infty \frac{d}{dt}\left[ \frac{e (\Delta(x) , t)}{b ( \gamma_t (\Delta(x) )  ) }\right]  g ( \gamma_t (\Delta(x) ) ) =: L + R.
\end{multline*}
Due to the support property of $g$ and of the fact that $e (\Delta(x), \infty ) = 0 ,$ by Assumption \ref{ass:2}, the left hand side equals
$$ L = - \frac{g( \Delta(x) ) }{b( \Delta(x))} ,$$
and the right hand side 
$$ R = \int_0^\infty e (\Delta(x) , t) \left[ \frac{f ( \gamma_t ( \Delta(x) ) }{b( \gamma_t (\Delta(x) ) } + \frac{ b' ( \gamma_t (\Delta(x) ) }{b( \gamma_t (\Delta(x) ) }\right] g ( \gamma_t (\Delta(x) ) )  dt .$$
Write $S_g$ for the support of $g. $ Since $ | b (\cdot ) | \geq \varepsilon $ on $S_g,$ we can upper bound
$$ | L | \le \frac{\| g\|_\infty}{\varepsilon}  \mbox{ and } | R | \le \frac{ \sup_{ x \in S_g }( f(x) + | b' (x) |)  }{\varepsilon  } \| g \|_\infty  \int_0^\infty e (\Delta(x) , t) dt .$$
Coming back to \eqref{eq:dim1}, we obtain
\begin{equation}\label{eq:IPPdim1}
 | m ( g') | \le C( \varepsilon) \| g\|_\infty \left( 1 + \int m(dx) f(x)  \int_0^\infty e (\Delta(x) ,  t) dt  \right) = 2 C( \varepsilon) \| g\|_\infty  ,
\end{equation}
since $ \int m(dx) f(x)  \int_0^\infty e (\Delta(x),  t) dt  = m(1) = 1.$ It is well known that \eqref{eq:IPPdim1} implies the existence of a continuous Lebesgue density of $m$ locally on $ {\mathcal A}, $ see e.g.\ Theorem 8 of  Bally and Caramellino (2011) \cite{Bally-Caramellino}. 
\end{proof}

Of course, in the multidimensional case, the above approach does not apply any more, since the noise present in one single jump event is not enough to generate $N-$dimensional noise in any direction of the space. We will show in Section \ref{sec:3} below how to use $N$ successive jump times in order to create a Lebesgue density also in dimension $N.$ But before doing so, we continue the above investigation and show how to obtain at least some regularity properties of the invariant density of a single particle within the configuration, based on the noise within the jump times.

\subsection{Smoothness of the invariant density of a single particle}\label{sec:23}
We exploit \eqref{eq:useful} to prove the regularity of the invariant density of a single particle when the whole system evolves in dimension $N.$ This regularity will be expressed explicitly depending on the smoothness of the underlying jump rate functions $f_i$ and the underlying drift vector $b.$ We work under the following additional assumptions.

\begin{ass}\label{ass:5}
1. ``No interactions in the flow'' : The drift vector is given by $b (x) = (\tb ( x^1 ) , \ldots , \tilde b ( x^N) ) $ where $\tilde b : \R \to \R$ is at least one times differentiable having a bounded derivative. \\
2. The jump functions are given by $ a_i (x) = (a_i^1 (x^1) , \ldots , a_i^{i-1} ( x^{i-1} ) , - x^i , a_i^{i+1} (x^{i+1} ) , \ldots , a_i^N (x^N) ) ,$ for all $ 1 \le i \le N,$ where all $ a_i^j : \R \to \R$ are infinitely differentiable having bounded derivatives of any order. Moreover, for all $ i \neq j , $ $ \R \ni v \to \Delta_i^j  (v) := v + a_i^j (v) $ is invertible with 
$$ \inf_{v \in \R} | \frac{d \Delta_i^j (v)}{d v } |  \geq a > 0$$
Finally, we suppose that 
$$ A := \max_{i \neq j } \sup_{k \geq 0} \sup_{v \in \R } | \frac{ d^k a_i^j (v)}{d v^k }  | < \infty .$$
3. For all $ 1 \le i \le N, $ $f_i ( x) = f_i (x^i ) , $ for all $x \in \R^N .$ 
\end{ass}

In this case, we can introduce the marginal flow for any single particle which is given by $ \tilde \gamma_{s, t } ( v) \in \R ,$ for $ 0 \le s \le t, $  solution of 
$$ \tilde \gamma_{s, t }  (v) = v + \int_s^t \tb ( \tilde \gamma_{s, u }(v) ) du .$$  
We write $ \tilde \gamma_t (v) := \tilde \gamma_{0, t } ( v) $ for the flow starting from $ v$ at time $0.$
Moreover, for any fixed $ v\in \R , $ we write $ \tilde \gamma^+ ( v) = \{ \tilde \gamma_t ( v) , t \geq 0 \}. $ 
We write ${\mathcal E} = \{ v^*  : \tb ( v^* ) = 0 \} $ for the set of  equilibrium points of  $\tilde \gamma $ and observe that, due to the linear growth property of $\tb, $ if $ v \notin {\mathcal E}, $ then also $ \tilde \gamma^+ ( v) \cap {\mathcal E } = \emptyset.$ (This is certainly well-known in the theory of one-dimensional dynamical systems, but we provide a short proof in the Appendix, see Section \ref{sec:flow}.)

We consider the following H\"older classes for jump rate functions and the drift function, for arbitrary constants $ F,B  > 0.$ 
\begin{equation}
 H( k ,F) = \{ f \in C^k ( \R , \R_+) : | \frac{d^l}{d v^l } f(v) | \leq F , \mbox{ for all }  0 \le l \le k, v \in \R  \} 
\end{equation} 
and 
\begin{equation}
 H( k+1 ,B) = \{ \tb \in C^{k+1} ( \R , \R ) : | \frac{d^l}{d v^l } \tb (v) | \leq B , \mbox{ for all }  0 \le l \le k+1 , v \in \R \} . 
\end{equation}

Our study is based on the following considerations. We start with the representation of the marginal law of the first particle in the invariant regime
\begin{equation} E_m ( g' ( X_t^1))  =\sum_{i=1}^N \int_{\R} m ( dx)  f_i ( x) \int_0^\infty  e^{ - \int_0^t  \bar f ( \gamma_s ( \Delta_i ( x) ) ds ) }    g'( \tilde \gamma_t (\Delta^1_i ( x)  ) ) dt ,
\end{equation}
where $m$ is the invariant measure of the system, for any smooth test function $g : \R \to \R. $ We then use  integration by parts with respect to $t$ within the integral expression 
\begin{multline*}
\int_0^\infty  e^{ - \int_0^t\bar f ( \gamma_s ( \Delta_i ( x) ) ds ) }    g' ( \tilde \gamma_t (\Delta^1_i ( x)  ) ) dt \\
= \left[ e^{ - \int_0^t \bar f ( \gamma_s ( \Delta_i ( x) ) ds ) }    \frac{g ( \tilde \gamma_t (\Delta^1_i ( x)  ) )}{\tb ( \tilde \gamma_t ( \Delta^1_i ( x) ) ) } \right]_{t= 0 }^{t= \infty } \\
- \int_0^\infty \frac{d}{dt} \left[ \frac{  e^{ - \int_0^t \bar f ( \gamma_s ( \Delta_i ( x) ) ds ) }}{{\tb ( \tilde \gamma_t ( \Delta^1_i ( x) ) ) } }\right] g ( \tilde \gamma_t ( \Delta^1_i ( x) ) ) dt ,
\end{multline*}
i.e.\ we exploit the smoothness of the flow as a function of time. In some sense, in doing so, we are close to the approach using the weak H\"ormander condition in diffusion theory, since we work with the drift of the system. This is why we will have to restrict our study to parts of the state space which are sufficiently far away from ${\mathcal E}, $ the set of all equilibrium points of the flow.  Moreover, the integration by parts gives rise to two border terms. The term corresponding to $ t = \infty $ disappears since $e^{ -\int_0^\infty \bar f ( \gamma_s (y) ds } = 0 , $ for all fixed $y.$  But the term corresponding to $t= 0 $ does not necessarily disappear and is given by 
$$  - \frac{ g ( x^1 + a^1_i ( x) ) }{\tb ( x^1 + a^1_i(x) ) }.$$
If $i = 1, $ i.e.\ if particle $1$ has just jumped, and if the transitions are given as in \eqref{eq:afterjump1}, then the above expression equals 
$$  - \frac{ g ( x^1 + a^1_1 ( x) ) }{\tb ( x^1 + a^1_1(x) ) } =  - \frac{ g ( 0 ) }{ \tb ( 0 ) },$$
creating a Dirac measure in $0 .$ The only way to prevent this fact is to suppose that $ g(0) = 0 .$ 

To summarize this discussion, we have to stay away from equilibrium points and from $0.$ It is for this reason that we restrict our study to the following open set defined by 
$$ S_{d , k+2 } = \{ v \in \R : (k+2)  A  < |v|  ,\;  | \tb ( v) | > d  \} ,$$
where $k$ is the smoothness of the fixed classes $ H( k, F)  $ and $ H( k+1, B),$ where $A$ comes from Assumption \ref{ass:5} and where $ d$ is such that $d >( k+2) A B  .$

\begin{theo}\label{theo:invmeasure}
Grant Assumptions \ref{ass:1}--\ref{ass:3} and \ref{ass:5}. Suppose that $f_i \in H( k , F) , $ for all $ 1 \le i \le N,$ and that $ \tb \in H(k+1 , B) .$ Fix some $1 \le i \le N$ and let 
$$\pi   := {\mathcal L}_m  ( X_t^i) $$
be the marginal law of the $i-$the particle in the invariant regime, i.e.\ $ \int g d  \pi = E_m ( g( X^i_t)   ) .$ 
Then $\pi $ possesses a bounded continuous Lebesgue density $p_\pi $ on $S_{d , k +2 } $ for any $d$ such that $d > (k+2)  A B .$ The density $ p_\pi $ is bounded on $S_{d , k +2  } ,$ uniformly in $f_i \in H( k ,F)$ and $ \tb \in H( k+1, B) .$  
Moreover, $p_\pi \in C^k ( S_{d , k +2 })  $ and 
$$ \sup_{\ell \le k , v \in S_{d , k+2 } } |  p_\pi^\ell ( v)  | + \sup_{v \neq v' , v, v' \in S_{d, k+2 }} \frac{p_\pi^{(k)} (v) - p_\pi^{(k)} (v') }{|v-v'|^\alpha } \le C ,$$ 
where the constant $C $ depends on $d,$ on $A$ and $a,$ and on the smoothness classes $ H( k, F )$ and $ H(k+1 , B),$ but on nothing else. 
\end{theo}

\begin{rem}
The above assertion remains true replacing $S_{d, k+2} $ by any set $S ( k+1, \varepsilon ) $ such that for all $0 \le m \le k+1 , $ for all choices $ i_1, \ldots , i_m \in \{ 1, \ldots , N \} \setminus \{ i \} , $ 
\begin{equation}
(\Delta_{i_1}^i \circ \ldots \circ \Delta_{i_m}^i)^{-1} ( S ( k+1, \varepsilon )  ) \subset \{ v \in \R : |\tb ( v) | > \varepsilon , |v| > \varepsilon \} .
\end{equation}
\end{rem}

The proof of this theorem follows the same ideas as those used in the proof of Theorem \ref{theo:4}. It is given in the Appendix.

\section{Lebesgue density in dimension $N$}\label{sec:3}
In the present section,  we come back to the study of the invariant measure $m$ of the whole particle system $ X_t = ( X_t^1, \ldots , X_t^N) .$ As argued before, it is sufficient to consider the jump chain $ Z_k = X_{T_k - } .$ For this jump chain, we show how several, typically $N,$ ``favorable'' transitions can create Lebesgue density in dimension $N,$ despite the partial degeneracy of the transition kernel $Q.$ In order to do so, we introduce skeletons for our process.

\subsection{Notations}
To fix notation, for any $n, m $ and any smooth function $ f : \R^n \to \R^m , $ given by 
$$  f = \left(
\begin{array}{l}
 f^1 ( x^1, \ldots , x^n ) \\
\vdots \\
f^m ( x^1, \ldots , x^n ) 
\end{array}
\right) , $$ 
we shall write 
$$ \frac{\partial f}{\partial x } = \left( \frac{ \partial f^i (x) }{\partial x^j } \right)_{ 1 \le i \le m , 1 \le j \le n }  \in \R^{ m \times n }, $$
which is the Jacobian matrix of $f.$  We shall also use the notation $ \check x^i = (x^1, \ldots , x^{i-1}, x^{i+1}, \ldots , x^n) $ for any $ 1 \le i \le n$ and $x = ( x^1 , \ldots , x^n) \in \R^n.$  

\subsection{Skeletons}
Fix $ n \geq 1 $ and let $ t_1, \ldots , t_n , t_{n+1} \in \R_+ $ be a succession of jump times and $ i_0, \ldots , i_n \in \{ 1, \ldots , N\} $ a succession of indices of jumping particles. We shall write shortly $ \tt = ( t_1, \ldots , t_{n+1}  ) $ and $ \ii = ( i_0, \ldots , i_n ).$ 
Finally, we write $ s_k = t_1 + \ldots + t_k , $ for any $ 1 \le k \le n +1.$ 

For any fixed $y \in \R^N, $ we introduce the sequence of configurations $y_k  = y_k (t_1^k ),  1 \le k  \le n ,$ and $ x_k = x_k ( t_1^k) , 1 \le k \le n ,$ by
\begin{equation}\label{eq:afterjump}
y_0 = y ,  x_0 = \Delta_{i_0 } (y) , \; y_{k+1} =\gamma_{t_{k+1}} (x_k )  ,  x_{k+1} =  \Delta_{i_{k+1}} ( y_{k+1}  ) , 1 \le k < n ,
\end{equation}
which are the possible positions of the process just before and just after a jump occurring at time $ s_{k+1}, $ by a particle with index $i_{k+1}.$  In the particular case $ t_k  = 0,$ we have $ x_k = \Delta_{i_k} ( x_{k-1} ) .$ We associate to this sequence the skeleton of our process 
$$ \eta_{x_0, \tt , \ii } (t) := 
\left\{
\begin{array}{ll}
\gamma_t ( x_0) ,& t < t_1 \\
\gamma_{t - s_k } ( x_k ) ,& s_k \le t < s_{ k +1 } , k < n \\
\gamma_{t - s_n } ( x_n ) ,& t \geq s_n 
\end{array}
\right\}  .$$ 
In particular, we will be interested in $ \eta_{x_0, \tt , \ii  } (s_{n+1} ) = \gamma_{t_{n+1} } ( x_n )  $ which is a possible configuration for $ Z_{n+1} , $ starting from $ Z_0 = y ,$ when we have imposed jumps at times $0, t_1 , \ldots , t_n$ by particles with indices $i_0, \ldots , i_n.$

Recall that $ \Delta_i ( x) = x + a_i (x) $ is the configuration after a jump of particle $i.$ Let
\begin{equation}\label{eq:A}
A^i ( x) = \left( \frac{\partial (\Delta_i^k ( x) )}{\partial x^l }(x) \right)_{ 1 \le k, l  \le N} .
\end{equation}
Introduce moreover $Y_t (x)  = \frac{\partial \gamma_t (x) }{\partial x } .$ Notice that $Y_t $ is solution of 
\begin{equation}\label{eq:Y}
Y_t ( x) = Id + \int_0^t \dot b ( \gamma_s (x) ) Y_s ( x) ds ,
\end{equation}
where 
$$
\dot b (x) = \frac{ \partial b }{\partial x} (x) =  \left( \frac{\partial b^k ( x) }{\partial x^l} \right)_{ 1 \le k , l \le N} .
$$
It is well known that under our conditions of linear growth on $ b, $ $Y_t (x) $ is invertible for all $ x \in \R^N $ and for all $ t \geq 0, $ having inverse matrix $ Z_t ( x) $ solution of 
\begin{equation}\label{eq:Z}
Z_t ( x) = Id - \int_0^t \dot b ( \gamma_s (x) ) Z_s ( x) ds .
\end{equation}

Then for $ 1 \le k \le n ,$ 
\begin{equation}\label{eq:abl1}
\frac{\partial \eta_{x_0, \tt , \ii  } (s_{n+1} )}{\partial t_k}  = 
\frac{ \partial \gamma_{t_{n+1} }}{\partial x }  (x_n )  \frac{\partial x_n }{\partial t_k}  = Y_{t_{n+1}} (n )  \frac{\partial x_n }{\partial t_k}  \mbox{ and }   \frac{\partial \eta_{x_0, \tt, \ii  } (s_{n+1} )}{\partial t_{n+1}} = b (\gamma_{t_{n+1}  } ( x_n ))  .
\end{equation}
Moreover, 
\begin{equation}\label{eq:abl2}
 \frac{\partial x_n}{\partial t_k} =  A^{i_n}  ( \gamma_{t_n} (x_{n-1} ) ) \frac{\partial \gamma_{t_n}}{\partial x } ( x_{n-1}) \frac{\partial x_{n-1} }{\partial t_k } = A^{i_n}  ( \gamma_{t_n} (x_{n-1} ) ) Y_{t_n}  ( x_{n-1}) \frac{\partial x_{n-1} }{\partial t_k } , 
\end{equation}
for all $ k < n , $ and
\begin{equation}\label{eq:abl3}
 \frac{\partial x_n}{\partial t_n} =A^{i_n}  ( \gamma_{t_n} (x_{n-1} ) )    b( \gamma_{t_n  } (x_{n-1} )) .
\end{equation}


\subsection{The derivation matrix}
We introduce the $N \times (n+1)  - $matrix 
\begin{equation}\label{eq:covariancematrix}
\sigma ({x_0,\tt , \ii  } ) := \left( \partial_{t_j} \eta^i _{x_0, \tt , \ii  } (s_{n+1} )\right)_{1 \le i \le N ,  1 \le j \le n+1}
= \left( 
\begin{array}{c}
(\nabla_t \eta^1_{x_0, \tt , \ii  } (s_{n+1} ))^T\\
\vdots\\
(\nabla_t \eta^N _{x_0, \tt , \ii} (s_{n+1} ))^T
\end{array}
\right). 
\end{equation}

Now, let $G$ be a smooth test function and fix $ n $ and $ \ii .$ Then a simple calculus shows that 
\begin{equation}\label{eq:change1}
 \nabla_t ( G \circ  \eta_{x_0, \tt, \ii}  (s_{n+1} ) ) =  (\sigma ({x_0, \tt , \ii } ) )^T  \; (\nabla_x G )  ( \eta_{x_0, \tt , \ii}  ( s_{n+1} ))  .
\end{equation}

Therefore we are interested in criteria ensuring that $ \sigma \sigma^T ( x_0, \tt, \ii ) $ is not degenerate. 
Using \eqref{eq:abl1}--\eqref{eq:abl3}, we obtain the following explicit representation of $\sigma ( x_0, \tt , \ii ) .$
\begin{equation}\label{eq:sigmaexplizit}
(\sigma ( x_0 , \tt, \ii ))^T  = 
\left( 
\begin{array}{l}
\left[ Y_{t_{n+1} } (x_n ) A^{i_n} ( \gamma_{t_n} ( x_{n-1} ) ) \cdots Y_{t_2} ( x_1) A^{i_1} ( \gamma_{t_1} ( x_0) ) b ( \gamma_{t_1} ( x_0 ) ) \right]^T \\
\left[ Y_{t_{n+1} } (x_n ) A^{i_n} ( \gamma_{t_n} ( x_{n-1} ) ) \cdots Y_{t_3} ( x_2) A^{i_2} ( \gamma_{t_2} ( x_1) ) b ( \gamma_{t_2} ( x_1 ) ) \right]^T\\
\vdots \\
\left[ Y_{t_{n+1} } (x_n ) A^{i_n} ( \gamma_{t_n} ( x_{n-1} ) )  b ( \gamma_{t_n} ( x_{n-1}  ) ) \right]^T\\
\left[b ( \gamma_{t_{n+1} } ( x_n ) ) \right]^T
\end{array}
\right) .
\end{equation}

This explicit form of $\sigma (x_0, \tt , \ii ) $ motivates the following definition.

\begin{defin}\label{def:good}
Fix $n \geq 1 $ and $ i_0, \ldots , i_n $ a sequence of indices. We introduce for all  $t_1, \ldots , t_{n+1} $ the following vector fields
$$ V_1( \tt, \ii)  = b ( \gamma_{t_{n+1} } ( x_n ) ) , V_2 ( \tt, \ii)  = Y_{t_{n+1} } (x_n ) A^{i_n} ( \gamma_{t_n} ( x_{n-1} ) )  b ( \gamma_{t_n} ( x_{n-1}  ) )  , \ldots , $$
$$ V_{n+1} (\tt, \ii) = Y_{t_{n+1} } (x_n ) A^{i_n} ( \gamma_{t_n} ( x_{n-1} ) ) \cdots Y_{t_2} ( x_1) A^{i_1} ( \gamma_{t_1} ( x_0) ) b ( \gamma_{t_1} ( x_0 ) ) ,$$
where $x_0, \ldots , x_n $ are chosen as in \eqref{eq:afterjump}.

We say that $ ( n +1 , \tt, \ii ) $ is {\bf good} if $ \{ V_1 ( \tt , \ii ) , \ldots , V_{n+1} ( \tt , \ii ) \}  $ spans $ \R^N $ 
 for all $ y \in \R^N.$ 
 \end{defin}

We given an example which is a system of interacting neurons, as considered in \cite{bresiliens} and \cite{pierre} where we can exhibit explicit sequences $ \ii $ such that $ (n+1, \tt , \ii ) $ is good for all $\tt .$  

\begin{ex}\label{ex:neuro}
In our example, $X_t^1, \ldots , X_t^N$ model the height of the membrane potential of $ N$ neurons. The model assumptions are as follows. For all $ 1 \le i \le N, $
let $ b^i  ( x) = - \lambda ( x^i - v^* ) , 1 \le i \le N, $ for some $ \lambda > 0 $ and $ v^* > 0 .$ This means that each membrane potential is attracted at exponential speed $ \lambda $ to a resting potential value $v^*.$  Let $ W_{i \to j } \in \R_+ $ for all $   1 \le i, j \le N ;  $ we interpret $ W_{i \to j } $ as synaptic weight of neuron $i$ on neuron $j.$ We  suppose that 
$$ x+ a_i (x) = \left( 
\begin{array}{c}
x^1 + W_{i \to 1 }  \\
\vdots \\
x^{i-1} + W_{i \to i-1} \\
0 \\
x^{i+1} + W_{i \to i+1} \\
\vdots\\
x^N + W_{i \to N} 
\end{array}
\right) , $$
i.e.\ $ a_i^j (x) = W_{i \to j } , $ for $ i \neq j, $ and $ a_i^i ( x) = - x^i.$ Let $n = N -1$ and take $i_k = k, 0 \le k \le N - 1.$ Then it is easy to see that, for fixed $ t_1 , \ldots , t_N,$  
$$ x_{i-1} = \left( 
\begin{array}{c}
W_{i \to 1} + \sum_{ k = 2}^{i-1} e^{ - \lambda ( t_k + t_{k+1} + \ldots + t_{i-1} ) } W_{ k \to 1 } + ( 1 - e^{ - \lambda (t_1 + \ldots + t_{i-1} ) } )v^*  \\
W_{i \to 2} + \sum_{ k = 3}^{i-1} e^{ - \lambda ( t_k + t_{k+1} + \ldots + t_{i-1} ) } W_{ k \to 2 } + ( 1 - e^{ - \lambda (t_2 + \ldots + t_{i-1} ) } )v^* \\
\vdots\\
W_{i \to i-1} + ( 1 - e^{ - \lambda t_{i-1} } ) v^* \\
0 \\
*\\
* \\
* 
\end{array}
\right) .$$ 
In particular, the only coordinate depending on $t_1 $ is the first one, the only two coordinates depending on $t_2 $ are the first two, and so on. It is then easy to see that the $k$th column of the derivation matrix is given by 
$$ \frac{\partial \eta_{x_0, t_1^{N} , i_0^{N-1} } (s_N)  }{\partial t_k } = \left( 
\begin{array}{c}
*\\
*\\
*\\
\lambda v^* e^{ - \lambda ( s_N - s_{k-1} ) } \\
0 \\
\vdots \\
0 
\end{array}
\right)  \longleftarrow \mbox{ coordinate } k .
$$ 
Therefore,
$$ \det \sigma ( x_0, \tt , \ii ) = \lambda^N (v^*)^N \prod_{k=1}^N e^{ - \lambda ( s_N - s_{k-1} ) }  \neq 0 $$
for all $ t_1, \ldots , t_N ,$ implying that for $ \ii =  ( 1, \ldots, N)  ,$  $ (N, \tt, \ii ) $ is good for all $\tt .$ Similar arguments apply for any $ \ii $ such that $ \{ i_0, \ldots , i_{N-1} \} = \{ 1 , \ldots , N \},$ i.e.\ each particle has jumped exactly once. 

\end{ex}

\subsection{Absolutely continuous parts of the invariant measure}\label{sec:34}

In general, it is difficult to check whether $ (n+1, \tt , \ii ) $ is good, since one has to ``solve'' explicitly the flow for all $t_1, \ldots , t_{n+1}, $ which leads to non-local criteria. For small $ t_1, \ldots , t_{n+1} , $ we can get rid of the flow in the following way. Let $ \bar x_0 = \Delta_{i_0} ( y), \bar x_1 = \Delta_{i_1} (\bar x_0 ) = \Delta_{i_1} \circ \Delta_{i_0 } (y) , \ldots , \bar x_n = \Delta_{i_n } ( \bar x_{n-1} )= \Delta_{i_n } \circ \ldots \circ \Delta_{i_0 } (y)  .$ This is the sequence of successive configurations introduced in \eqref{eq:afterjump}, for $ t_1 = \ldots = t_n = 0 .$ 

\begin{prop}
Introduce the vector fields 
$$ V_1 ( \ii ) = b ( \bar x_n) , V^2 ( \ii) =  A^{i_n } (\bar x_{n-1} ) b( \bar x_{n-1} ) , V^3 ( \ii) = A^{i_n } (\bar x_{n-1} ) A^{i_{n-1} } ( \bar x_{n-2}) b (\bar x_{n-2} ) ,\ldots $$  
and
$$  V_{n+1} ( \ii ) = A^{i_n } (\bar x_{n-1} ) \ldots A^{i_1} ( \bar x_0) b ( \bar x_0 ) .$$ 
Suppose that $ V_1 ( \ii ) , \ldots , V_{n+1} (\ii) $ span $\R^N $ for some $\ii.$ Then $(n+1, \tt , \ii ) $ is good for $ t_1, \ldots , t_{n+1} $ small enough.
\end{prop}

\begin{proof}
The proof follows from the continuity of $ \tt \mapsto \det \sigma \sigma^T ( x_0, \tt , \ii ) $ and the fact that, by definition, $ \det \sigma \sigma^T ( x_0, \0 , \ii ) > 0 ,$ where $ \0 $ denotes the sequence of successive times $t_1 = \ldots = t_{n+1} = 0 .$ 
\end{proof}

Recall the definition of $Z_k = X_{T_k - } $ and write $K$ for its transition kernel and $K^n $ for its $n-$fold iteration.  

\begin{theo}\label{theo:7}
Grant Assumptions \ref{ass:1}--\ref{ass:3}. 
Suppose that there exist $n$ and $\ii $ such that 
$$ \inf_{x_0 \in \R^N} \det \sigma \sigma^T ( x_0, \0 , \ii )> 0,$$
i.e.\ $V_1 (\ii) , \ldots , V_{n+1} ( \ii) $ span $ \R^N,$ uniformly in $x_0.$ Then the following Doeblin type lower bound holds. For all $z_0$ there exist $z_n \in \R^N ,$ $ \delta_1 , \delta_2  > 0 $ and $ \beta \in ]0, 1 [ $  such that  
\begin{equation}\label{eq:LD}
K^{n+1}  ( x, dy  ) \geq \beta 1_C (x)  \nu (y)  (dy),
\end{equation}
where  $ C = B_{\delta_1} ( z_0 ) $ and $\nu 
$ is a smooth probability density having compact support within $ B_{\delta_2}  ( z_{n}) .$  
\end{theo}

\begin{proof}
Fix $ \varepsilon > 0 $ sufficiently small such that $ (n+1, \ii , \tt ) $ is good for all $ \tt $ with $ t_1, \ldots , t_{n+1} \le \varepsilon .$  Let 
$$ E = \{ T_1 \le \varepsilon, \ldots ,  T_{n+ 1}- T_{n} \le \varepsilon ,   I_0 = i_0, \ldots, I_{n} = i_n \},$$
where $ I_k $ denotes the index of the jumping particle at time $ T_k, k \geq 0.$ Then we have for any measurable $ B \in {\mathcal B } ( \R^N ) $ and for any fixed $x,$  
\begin{multline*}
K^{n+1} ( x, B)  = P_x ( Z_{n+1} \in B ) \geq P_x ( Z_{n+1} \in B , E ) =  \frac{f_{i_0} ( x)}{\bar f (x) }  \int \delta_{ \Delta_{i_0} ( x )} ( d x_0 ) \\
\int_0^\varepsilon e( x_0, t_1) 
 f_{i_1} ( \gamma_{t_1} (x_0) ) dt_1 \int \delta_{\Delta_{i_1} ( \gamma_{t_1} (x_0) )} (dx_1)  \ldots \int_0^\varepsilon e( x_{n-1}   , t_{n}) f_{i_{n}} ( \gamma_{t_{n}} (x_{n-1}) ) dt_{n}  \\
\int \delta_{\Delta_{i_{n}} ( \gamma_{t_{n} } (x_{n-1})) } ( d x_{n} ) \int_0^\varepsilon \bar f ( \gamma_{t_{n+1} } ( x_{n} ) )  e ( x_{n} ,  t_{n+1} ) 
 1_B ( \gamma_{t_{n+1} } ( x_{n} )) dt_{n+1} \\
=  \frac{f_{i_0} ( x)}{\bar f( x) }  \int \delta_{ \Delta_{i_0} ( x )} ( d x_0 )  
 \int_0^\varepsilon d t_1 \ldots \int_0^\varepsilon d t_{n+1} q( x_0, \tt, \ii ) 1_B ( \eta_{ x_0, \tt, \ii } ( s_{n+1} )) ,
\end{multline*}
where $ \tt = (t_1, \ldots , t_{n+1} ) , $ $\ii = (i_0, \ldots , i_n ) ,$ and where 
$$ q( x_0, \tt, \ii )  =  \prod_{k=1}^{n} [e( x_{k-1}, t_{k} ) f_{i_{k} } ( \gamma_{t_k} (x_{k-1} )) ] \bar f ( \gamma_{t_{n+1} } ( x_n) )e( x_{n}, t_{n+1}).$$
Under our conditions, the mapping
$$ \frac{f_{i_0} ( x)}{\bar f ( x) }   q( x_0, \tt, \ii ) $$
is strictly lower bounded on $ x \in B_{\delta_1} ( z_0 ) , t_1, \ldots , t_{n+1} \le \varepsilon, $ for any fixed $ \delta_1  > 0 $ and $z_0.$ 

Moreover, under our assumptions, 
$$ \R_+^{n+1} \ni \tt \mapsto \eta_{ z, \tt, \ii } ( s_{n+1} ) ,$$
for any fixed $ z \in \R^N,$ is a submersion at $ \tt $ for any $ \tt $ with $t_1, \ldots , t_{n+1} \le \varepsilon, $ since the partial derivatives with respect to $t_1, \ldots , t_{n+1} $ span $\R^N .$ Put $ \tt_0 = ( \varepsilon, \ldots , \varepsilon) .$ For any fixed $ x= z_0 $ and $ x_0 = \Delta_{i_0 } ( z_0),$ write $z_{n} :=   \eta_{ x_0, \tt_0, \ii } ( \varepsilon ) .$ Then we can apply Theorem 4.1 and Lemmata 6.2 and 6.3 of Bena\"{i}m et al. (2015) \cite{micheletal}. They imply that there exist $ \delta_1 , \delta_2 > 0 $ such that 
$$ K^{n+1}  ( x, dy  ) \geq \tilde \beta 1_C (x)  1_{ B_{\delta_2} ( z_{n+1} )}(y)   (dy) ,$$
for $ C = B_{\delta_1}  ( x_0).$ Choosing a $C^\infty -$function $\tilde \nu $ 
such that $ 0 \le \tilde \nu  \le 1_{ B_{\delta_2} ( z_{n} )} , $ $ \int \tilde \nu > 0 , $ and putting $ \nu := (\int \tilde \nu ) ^{-1} \tilde \nu $  implies the desired result \eqref{eq:LD}. 
\end{proof}

\subsection{Nummelin splitting and creation of Lebesgue density for $Z_k $}\label{sec:35}
We will use the above Doeblin lower bound to introduce a splitting procedure which is inspired by the so-called {\it Nummelin splitting} introduced by Athreya and Ney (1978) \cite{athreyaney} and Nummelin (1978) \cite{nummelin}. First of all, since $(Z_k)_{k \geq 1 }$ is Harris recurrent, we may choose $z_0 $ such that $ m^Z (B_{\delta_1} ( z_0 ) ) > 0 .$ As a consequence, $ (Z_k)_k$ visits $ C = B_{\delta_1} ( z_0 ) $ infinitely often. So let 
$$ S_1 = \inf \{ n : Z_n \in C \} , \dots , S_{k+1} = \inf \{ n > S_k : Z_n \in C \} $$
be the successive visits of the set $C.$ We have $S_k < \infty $ almost surely for all $k.$ Let $(U_n)_n $ be a sequence of i.i.d.\ uniform random variables, uniformly distributed on $ [0, 1], $ independent of $ (Z_k)_k .$ Then \eqref{eq:LD} allows to write 
\begin{equation}\label{eq:splitting}
 Z_{ S_k + n+1}  \stackrel{\mathcal L}{=} 1_{ [0, \beta ] } ( U_k ) Y_k + 1_{ ] \beta , 1] } ( U_k ) W_k   , 
\end{equation} 
where $ (Y_k)_k $ are i.i.d.\ random variables, distributed $\sim \nu , $ independent of $ (Z_k)_k, $ and where 
$$ W_k \sim \frac{ 1}{1- \beta } \left( K^{n+1}  ( Z_{S_k} , dy  ) - \beta \nu (y)  ( dy )  \right) .$$ 
We will therefore use the representation \eqref{eq:splitting} and introduce the regeneration time 
$$ R = \inf \{ S_k + n +1 :  k \geq 1 ,  U_k \le \beta \} .$$
Since $(Z_k)_k $ visits $C$ infinitely often almost surely, clearly, $R < \infty $ almost surely. 
Then \eqref{eq:splitting} reads as follows : 
\begin{equation}\label{eq:splitting2}
Z_{R} \stackrel{\mathcal L}{=} \nu (y) dy  .
\end{equation}

Therefore, we have proven the following. 

\begin{cor}[Nummelin splitting and creation of Lebesgue density for $Z_k $]\label{cor:1}
Grant the conditions of Theorem \ref{theo:7}. Then there exists an extended stopping time $R$ with $R < \infty $ almost surely and such that $ {\mathcal L} (Z_{R }) $ is absolutely continuous with respect to the Lebesgue measure on $ \R^N $ having smooth density $ \nu \in C^\infty .$  
\end{cor}

It remains to prove that the smoothness created at the regeneration time $R $ is preserved, under suitable conditions, by the dynamics.  This is far from being obvious, due to the degenerate structure of the transition kernel $Q.$ However, in some cases, we are at least able to show that Lebesgue absolute continuity is preserved. 

\subsection{Preservation of Lebesgue absolute continuity}\label{sec:36}
We intend to find conditions implying that if  $Z_0 \sim p(x) dx  , $ for some measurable $p,$ then $ {\mathcal L} ( Z_{1} ) $ is also absolutely continuous with respect to the Lebesgue measure. 
Let $ g$ be a smooth test function. We write $ E_p$ for the conditional expectation, given $ Z_0 \sim p (x) dx.$ Then 
\begin{equation}\label{eq:transition}
E_p ( g ( Z_1) ) = \sum_{i=1}^N \int p(x)  \frac{f_i ( x)}{\overline f (x) }  dx  \int_0^\infty e( \Delta_i ( x) , t ) \bar f ( \gamma_t (\Delta_i ( x) ) )  g  ( \gamma_t (\Delta_i (x) ) ) dt .
\end{equation}
As explained in Example \ref{ex:neuro}, we are interested in transitions of the kind 
$$ a_i ( x) = \left( 
\begin{array}{c}
a^1_i (x^1)\\
\vdots \\
a^{i-1}_i ( x^{i-1} )\\
- x^i \\
a^{i+1}_i ( x^{i+1} ) \\
\vdots\\
a^N_i ( x^N) 
\end{array} 
\right) , \;  \Delta_i ( x) = x + a_i (x) = 
\left( 
\begin{array}{c}
x^1 + a^1_i (x^1)\\
\vdots \\
x^{i-1} + a^{i-1}_i( x^{i-1} )\\
0\\
x^{i+1} + a^{i+1}_i ( x^{i+1} ) \\
\vdots\\
x^N + a^N_i ( x^N) 
\end{array} 
\right) .
$$
$ \Delta_i ( x) $ does not depend on $x^i $ any more and has $0-$entry in the $i-$th coordinate. Therefore,  also $ \gamma_t ( \Delta_i ( x) ), $ which is the evolution of the flow after a jump of the $i-$th particle, does not depend on $x^i .$ So even if we start with an $N-$dimensional density $ p ( x) $ as in \eqref{eq:transition}, after a jump of the $i-$the particle, there is no density in direction of $e_i,$ the $i-$th unit vector of $\R^N, $ any more. 

The main idea is to replace this missing direction by the noise which is created by the jump times, i.e.\ to use the additional noise created by $t$ in \eqref{eq:transition}. This strategy works if the noise created by the exponential jump times has non zero component in direction of $ e_i.$ The following theorem relies on this idea and the coarea formula (we refer to Federer (1996) \cite{federer}) which allows to make a simple change of variables. The theorem  implies the absolute continuity of the invariant measure $m$ but does not give any regularity of the invariant density.

\begin{theo}\label{theo:8}
Grant Assumptions \ref{ass:1}--\ref{ass:3} and suppose that there exist $n$ and $\ii $ such that 
$$ \inf_{x_0 \in \R^N} \det \sigma \sigma^T ( x_0, \0 , \ii )> 0.$$
Suppose moreover that for all $ 1 \le i \le N, $ for all $ t \geq  0  $ and $x \in \R^N,$  the columns of $ Y_ t ( \Delta_i (x) ) A^i (x) $ and $ b ( \gamma_t ( \Delta_i ( x) ) $ span $\R^N.$ Then the invariant measure $m$ of the process $(X_t)_{t \geq 0} $ is absolutely continuous with respect to the Lebesgue measure.  
\end{theo}

\begin{rem}
In the non-interacting case of Assumption \ref{ass:5}, the columns of $ Y_ t ( \Delta_i (x) ) A^i (x) $ and  the vector $ b ( \gamma_t ( \Delta_i ( x) ) $ span $\R^N $ if and only if $ 0 \notin {\mathcal E}, $ see Step 1.\ in the next Section \ref{sec:37}.
\end{rem}

\begin{proof} 
\hspace{0,1cm}

{\bf Step 1.} We rely on \eqref{eq:transition} and study, for any fixed $i,$ the mapping $ (x^1 , \ldots , x^N , t ) \mapsto \gamma_t ( \Delta_i ( x) ) := G( t, x) .$ Its $N-$dimensional Jacobian is given by 
$$ J_G ( t, x) = \sqrt{ \det \left( \frac{\partial G (t,x) }{\partial t \partial x} (\frac{\partial G (t,x) }{\partial t \partial x})^T\right)} .$$

But the first column of $\frac{\partial G (t,x) }{\partial t \partial x} ,$ the partial derivative with respect to time, is given by $ b ( \gamma_t (\Delta_i ( x) ) ) .$ The following columns of of $\frac{\partial G (t,x) }{\partial t \partial x} $ are precisely the columns of $Y_ t ( \Delta_i (x) ) A^i (x) .$  These columns span $ \R^N, $ by our assumptions. As a consequence, $J_G ( t, x) > 0 $ for any fixed $ (t,x) .$ We use the coarea formula and obtain, for any smooth test function $g  : \R^N \to \R_+, $ putting $ W_i ( t, x) = p(x) \frac{f_i ( x)}{\overline f (x) }  e( \Delta_i ( x), t) \bar f \circ G ( t, x) / J_G ( t, x) ,$  
\begin{multline*}
E_p ( g ( Z_1) ) = \sum_i \int_{\R^N } dx \int_0^\infty dt \; W_i ( t, x)  g  ( G(t, x)) J_G ( t, x) \\
=\sum_i \int_{\R^N } \left( \int_{ G^{-1} ( x) } W_i (z) g \circ G (z)   {\mathcal H}_{1} ( dz) \right) dx \\
= \sum_i \int_{\R^N } g ( x) \left( \int_{ G^{-1} ( x) } W_i (z)    {\mathcal H}_{1} ( dz) \right) dx ,
\end{multline*}
where $  {\mathcal H}_{1} $ is the one-dimensional Hausdorff measure and $G^{-1} ( x) = \{ (t, z) : G(t, z) = x \} .$

This shows that if $ Z_n $ is absolutely continuous with respect to the Lebesgue measure, having density $p, $ then $ Z_{n+1} $ is also absolutely continuous having density 
$$ q (x) = \sum_i  \int_{ G^{-1} ( x) } W_i (z)    {\mathcal H}_{1} ( dz) .$$ 
This implies a fortiori that $Z_{n+k }$ is absolutely continuous for all $k \geq 1$. 

{\bf Step 2.} By \eqref{eq:splitting2}, ${\mathcal L} (Z_{R }) = \nu (x) dx   $ is absolutely continuous, and $R  < \infty $ almost surely.  We introduce $ R_1 := R, R_{n+1} := R_n + R_1 \circ \theta_{R_n} $ for all $ n \geq 1,$  where $ \theta $ denotes the shift operator on the space of trajectories of $ (Z_n) _n.$ Then, by the Kac occupation formula, for any bounded test function $ g  : \R^N \to \R, $ 
$$ m^Z ( g ) = \frac{1}{ E ( R_2- R_1) }  E \left( \sum_{k=R_1}^{R_2- 1 } g ( Z_k ) \right) = \frac{1}{ E ( R_2- R_1) } \sum_{k \geq 0} E_\nu \left( g ( Z_k ) 1_{\{ k < R_1 \}} \right) ,$$
which implies that $ m^Z < \! \! < \lambda , $ and therefore also $m < \! \! < \lambda  .$ 
\end{proof}

\begin{rem}
It is important to stress here that we did not use an integration by parts formula in the above proof. 
\end{rem}

Theorem \ref{theo:8} gives the absolute continuity of the invariant measure without any further smoothness properties of the invariant density. When imposing more structure on the dynamics of the process, we are able to obtain finer results as we will show in the next subsection.

\subsection{Discussion of  the non-interacting case}\label{sec:37}
We will suppose within this subsection that there are no interactions between coexisting particles within the dynamics in the flow, i.e.\ the only interactions are produced by the jumps, as in Assumption \ref{ass:5}. Recall that we denote in this case by $ \tilde \gamma_t (v) $ the marginal flow of a single particle issued from $ v \in \R ,$ and that ${\mathcal E} = \{ v^*  : \tb ( v^* ) = 0 \} $ is the set of  equilibrium points of this flow. 

\begin{ass}\label{ass:9}
$ 0 \notin {\mathcal E} .$ 
\end{ass}
Under Assumption \ref{ass:9}, $ [0, \infty [ \ni t \mapsto \tilde \gamma_t ( 0 ) \in \R \setminus { \mathcal E} $ is invertible, and we write $ \kappa : \tilde \gamma^+ ( 0 )  \to [0,  \infty [ $  for its inverse.

We impose Assumptions \ref{ass:1}--\ref{ass:3}, Assumption \ref{ass:5} and Assumption \ref{ass:9}. We suppose moreover that  $f_i \in H(k, F) ,$ for all $ 1 \le i \le N,$ and $\tb \in H(k , B) ,$ for all $ k \geq 1 . $ We write $ {\mathcal A} = \{ x \in \R^N :  x^i \notin {\mathcal E}  , \; \; \forall 1 \le i \le N \} .$

{\bf Step 1.} First of all, by the structure of our dynamics, 
$$ A^i ( x) = diag \left( 1 +  (a_i^1)' (x^1 ), \ldots, 1 + ( a_i^{i-1})' (x^{i-1} ), 0, 1 + ( a_i^{i+1})' (x^{i+1} ), \ldots , 1 + ( a_i^{N})' (x^{N} ) \right) ,$$
where $ (a_i^j)' (v) = \frac{d a_i^j (v)  }{d v} , $ for any $ v \in \R.$ By Assumption \ref{ass:5}, any of the elements $ |1 + ( a_i^j )'(v)|  \geq a > 0 .$ Moreover, since there are no interactions in the dynamics of the flow, 
\begin{equation}\label{eq:Y}
 Y_t (x) = diag \left( y_t (x^1 ) , \ldots , y_t (x^N) \right) 
, \mbox{where 
$y_t ( v) = 1 + \int_0^t \tb' ( \tilde \gamma_s (v)) y_s (v) ds .$} 
\end{equation} 
Introducing 
\begin{equation}\label{eq:Z}
 z_t ( v) = 1 - \int_0^t \tb' ( \tilde \gamma_s ( v) ) z_s (v) ds,
\end{equation}
$Y_t (x) $ has then the inverse matrix
$
 Z_t (x) =  diag \left( z_t (x^1) , \ldots , z_t (x^N)   \right) .
$

This implies that $ ( N, \tt , \ii) $ is good for $ \ii = ( 1, 2, \ldots , N ),$ for any $ \tt .$ Thus, the conditions of Theorem \ref{theo:7} are fulfilled. Moreover, 
$$
 Y_t ( \Delta_i ( x) ) A^i (x) =
 diag  \left( 
w^1_t (x^1 ) ,     \ldots ,  w^{i-1} _t ( x^{i-1})  , 0 ,  w^{i+1} _t ( x^{i+1} )  , \ldots ,  w^N_t ( x^N) 
\right) ,
$$
where $ w^j_t (x^j ) = (1 +  (a_i^j)' (x^j )) y_t (\Delta_i^j  (x^j ) )  .$  Moreover, 
$$ b ( \gamma_t ( \Delta^i (x) ) = \left(
\begin{array}{c}
* \\
\vdots \\
* \\
\tb ( \tilde \gamma_t (0))  \\
* \\
\vdots \\
* 
\end{array} \right) 
\leftarrow \mbox{ $i-$th coordinate $\neq 0 $ },$$
implying that the columns of  $ Y_t ( \Delta_i ( x) ) A^i (x) $ and $   b ( \gamma_t ( \Delta^i (x) ) $ span $ \R^N $ for all $ t .$ 
Hence the conditions of Theorem \ref{theo:8} are fulfilled as well. We deduce from Theorem \ref{theo:7} that there exists an extended stopping time $R $ such that $ Z_{R } \sim \nu ( x) dx , $ with $ \nu \in C^\infty_c  (B_{\delta_2} (z_n) )$ for some $z_n \in \R^N.$  In the proof of Theorem \ref{theo:8} we have shown that this implies that the invariant density is absolutely continuous. In the following step we will study how the smoothness of the ``regeneration'' density $\nu$ is preserved by the dynamics. 

{\bf Step 2.} Suppose therefore that $ Z_{n}\sim \nu ( x) dx $  with $ \nu $ the ``regeneration density''.  We first show that the law of $ Z_{n + 1 } $ possesses also a Lebesgue density for which we can exhibit an explicit representation. 

In order to do so, we start with the following observation.

\begin{lem}
It is always possible to choose the regeneration density $ \nu $ of Theorem \ref{theo:7} such that $ \nu (x) / \bar f (x) = \prod_{i=1}^N r (x^i ) $ is of product form, with $r \in C^\infty_c ( \R , \R_+)  .$ 
\end{lem}

\begin{proof}
This simply follows from choosing $r\in C^\infty_c ( \R , \R_+) $ such that $ (\prod_{i=1}^N r (x^i ) )\bar f (x) \le 1_{ B_{\delta_2} (z_n) } (x) , $ as in the end of the proof of Theorem \ref{theo:7}. This choice is always possible since $ \bar f $ lower bounded on $B_{\delta_2} (z_n) .$ 
\end{proof}

Let $ g  $ be a smooth test function.
We condition on $ Z_{n } \sim \nu ( x) dx . $ Then, by \eqref{eq:transition},
\begin{equation}\label{eq:encore}
E (  g ( Z_{n+1} ) | Z_n \sim \nu ) = 
\sum_{i=1}^N \int_{\R^N}  \nu (x)  \frac{f_i ( x^i)}{\bar f (x) }  dx  \int_0^\infty e( \Delta_i ( x) , t ) \bar f ( \gamma_t (\Delta_i ( x) ) )  g ( \gamma_t (\Delta_i (x) ) ) dt .
\end{equation}

Notice that due to our assumptions, $ \Delta_i ( x) $ does not depend on $x^i .$ As a consequence, 
$$ \gamma_t ( \Delta_i ( x) ) = ( \tilde \gamma_t ( \Delta_i^1 ( x^1)), \ldots , \tilde \gamma_t ( \Delta_i^{i-1} (x^{i-1} )), \tilde \gamma_t ( 0) , \tilde \gamma_t ( \Delta_i^{i+1} (x^{i+1} )), \ldots , \tilde \gamma_t ( \Delta_i^N ( x^N) ) )$$ 
does not depend on $ x^i , $ neither. 

Fix $i.$ We work with fixed $t$ and use $N-1$ times the one-dimensional transformation of variables given by  
\begin{equation}
 y^j =  \tilde \gamma_{t} ( \Delta_i^j (x^j ) ) = : \beta (i \to j ,  t,   x^j ) , \; \frac{ d y^j }{dx^j } = y_{t  } ( \Delta_i^j ( x^j ) ) ( 1 + (a_i^j)' (x^j ) )  ,
\end{equation}
for any $ j \neq i .$ 

By the explicit equation for $ z_t (x) $ in \eqref{eq:Z} and the linear growth condition on $ \tb ,$ we obtain 
\begin{equation}\label{eq:controllambda}
| \frac{1}{y_{t  } ( \Delta_i^j ( x^j ) ) ( 1 + (a_i^j)' (x^j ) )  } | \le \frac{1}{a} | z_{t  } ( \Delta_i^j ( x^j ) ) | \le \frac{1}{a} e^{B t }.
\end{equation}
As a consequence, $ \R \ni v \mapsto \beta (i\to j ,  t, v)  =: \beta_{  i \to j, t } (v) $ is invertible. We write $\beta_{ i \to j , t } ^{-1} (\cdot )  $ for its inverse function. Notice that $ v \notin {\mathcal E} $ implies that $\beta_{ i \to j , t } ^{-1} (v) \notin {\mathcal E}.$ 

We obtain
$$ d x^j  =\lambda_{i \to j }  ( t, y^j ) dy^j ,$$
for all $ j \neq i , $ where 
\begin{equation}\label{eq:lambda}
\lambda_{i \to j }  (t,  y^j ) = \frac{1}{1 + (a_i^j)' ( \beta_{ i \to j , t } ^{-1}( y^j ) )} z_{t } ( \Delta_i^j ( \beta_{ i \to j , t } ^{-1}( y^j )))  , \;  \sup_{ v \in \R } | \lambda_{ i \to j } ( t, v) |  \le \frac{1}{a} e^{B t}.
\end{equation}
Once these $N-1$ transformations of variables done, we work at fixed $ y^1 , \ldots ,  y^{i-1} , y^{i+1} , \ldots y ^N $ and use the transformation of variables $ y^i = \tilde \gamma_t ( 0) ,$
$ \frac{dy^i}{dt } = \tb ( \tilde \gamma_t (0 ) ) = \tb (y^i ) \neq 0 $ due to Assumption \ref{ass:9}. Recall that
$$
\tilde \gamma^+ ( 0 )  \ni y \mapsto \kappa ( y ) 
$$
denotes the inverse function of $ t \mapsto \tilde \gamma_t (0 ) .$ 
 
We write for any $z \in \R^N, $ 
$$\check  z^i := ( z^1 , \ldots , z^{i-1}, z^{i+1}, \ldots , z^N )  \mbox{ and }  (\check z, y^i  ) := ( z^1, \ldots , z^{i-1} , y^i , z^{i+1} , \ldots , z^N).$$
Moreover, let 
$$ \beta^{-1}_{i \to \cdot } (y  ) := (\beta_{ i \to 1, \kappa ( y^i ) }^{-1}  ( y^1), \ldots , \beta_{ i \to i-1, \kappa ( y^i ) }^{-1} ( y^{i-1} ), \beta_{ i \to i+1, \kappa ( y^i ) }^{-1} ( y^{i+1} ), \ldots , \beta_{ i \to N , \kappa ( y^i ) }^{-1} ( y^{N} ) )$$
and 
$$ ( \beta^{-1}_{i \to \cdot }  ( y), z^i ) :=  ( \beta_{ i \to 1, \kappa ( y^i ) }^{-1}  ( y^1), \ldots , \beta_{ i \to i-1, \kappa ( y^i ) }^{-1} ( y^{i-1} ),  z^i , \beta_{ i \to i+1 , \kappa ( y^i ) }^{-1} ( y^{i+1} ), \ldots , \beta_{i \to N, \kappa ( y^i )  }^{-1} ( y^{N} ))  $$
for any $z^i \in \R.$ 
Then, coming back to \eqref{eq:encore} and resuming the above discussion, 
\begin{multline*}
\int_{\R^N}  \nu (x)  \frac{f_i ( x^i)}{\bar f (x) }  dx  \int_0^\infty e( \Delta_i ( x) , t ) \bar f ( \gamma_t (\Delta_i ( x) ) )  g ( \gamma_t (\Delta_i (x) ) ) dt\\
= \int_\R 1_{ \tilde \gamma^+ ( 0) } (y^i) dy^i \int_{\R^{N-1}}  d\check y^i   \\
\left[ \left\{ \int_{\R} dx^i  ( \nu / \bar f )  (( \beta^{-1}_{i \to \cdot} ( y ), x^i ))  f_i (  x^i ) \right\} \tilde e ( y)  \prod_{j \neq i } \lambda_{i \to j } (\kappa ( y^i ) ,  y^j)  \frac{\bar f ( y  ) }{|\tb ( y^i ) |} \right]  g ( y ), 
\end{multline*}
where 
$$ \tilde e ( y) = \exp \left( { - \int_0^{\kappa ( y^i )} \sum_j  f_j ( \tilde \gamma_{s, \kappa (y^i ) }^{ -1} (y^j) ) ds} \right) ,$$
since $ \tilde \gamma_s ( v) = \tilde \gamma_{s, t }^{-1} ( \tilde \gamma_t (v) ) ,$ for all $s \le t ,$ where $\tilde \gamma_{s, t }^{-1} $ denotes the inverse flow. 

Now we exploit the fact that we have chosen the regeneration density such that $ \nu (x) / \bar f (x) = \prod r(x^i ) . $ As a consequence,
$$ \int_{\R} dx^i  ( \nu / \bar f )  (( \beta^{-1}_{i \to \cdot} ( y ), x^i ))  f_i (  x^i ) = C( f_i , \nu) \prod_{j \neq i } r ( \beta^{-1}_{i \to j , \kappa (y^i ) } (y^j )) ,$$
where $ C(f_i, \nu ) = \int  r(x^i ) f_i (x^i ) dx^i.$ 

Hence we may introduce 

\begin{equation}\label{eq:qi}
q_i (y) := C( f_i , \nu) \prod_{j \neq i } \left[ r ( \beta^{-1}_{i \to j , \kappa (y^i ) } (y^j ))    \lambda_{i \to j } (\kappa ( y^i ) ,  y^j) \right]  \tilde e ( y)  \frac{\bar f ( y  ) }{|\tb ( y^i )| } 1_{ \tilde \gamma^+ (0) } (y^i )  .
\end{equation}
Notice that $ q_i (y) / \bar f (y) $ is once more of product form with respect to $y^j , $ for $j \neq i , $ once $y^i $ is fixed (the term $\tilde e (y) $ is also of product form). As a consequence, the density of $Z_{n+1} $ is the sum of densities such that each of them, divided by $\bar f,$ is of product form :

\begin{equation}
E (  g ( Z_{n+1} ) | Z_n \sim \nu ) =  \int \left( \sum_{i=1}^Nq_i (y)\right) g  (y)  dy .
\end{equation}

{\bf Step 3.} We have to study the regularity of the Lebesgue density $\left( \sum_{i=1}^Nq_i (y)\right) $ in terms of the regularity of the initial density $\nu.$ For that sake we introduce the following objects.

For any multi-index  $ \alpha = ( \alpha_1, \ldots , \alpha_m ) \in \{ 1 , \ldots , N \}^m $ and any fixed $ i ,$  we denote by $ \| \alpha \|_{(i)} := \sum_{ l=1}^m 1_{ \alpha_l = i } $ the number of times that a derivation with respect to $y^i $ arises. Moreover, we put $ \| \alpha \| = m .$ 

Let us then introduce the set of probability densities $ \mathcal P^{(\infty )} = \{ p ( x) dx , p : \R^N \to \R_+ , \int p (x) dx = 1 \}$ satisfying the following three points.
\begin{enumerate}
\item
$ p \in C^\infty ( {\mathcal A}) .$
\item
$ \check p_j (\check x^j ): = \int p ( x) (f_j ( x^j )/\bar f (x) )  dx^j $ belongs to $ C^\infty  (\{ x \in \R^{N-1} : x^i \notin {\mathcal E} , \forall i \}),$ for all $j .$ 
\item
For all  $ \ell \le  N-1, $ for all $ k_1< \ldots < k_\ell  \in \{ 1, \ldots , N \} \setminus \{j \} ,$  for all $ \alpha $ such that $ \| \alpha \|_{(k_1)} =\ldots =\| \alpha \|_{(k_\ell )}=  0,  $ $ \int_{\R^\ell}  | \partial_\alpha  \check p_j ( \check x^j ) |d x^{k_1} \ldots d x^{k_\ell }  < \infty .$ 
\end{enumerate}
Clearly, $ \nu \in {\mathcal P}^\infty.$  We are seeking for conditions under which also  $ q_i \in \mathcal P^\infty$ for all $ 1 \le i \le N .$ 

The points (2) and (3) of the definition of ${\mathcal P}^\infty$ are actually the points that may pose problems  : we are not a priori sure to stay away from $ {\mathcal E}$ where the densities will possibly explode, when integrating.

It is evident that $ q_i \in C^\infty ({\mathcal A})$ since all coefficients are smooth. Moreover, we have the obvious upper bound
\begin{equation}\label{eq:qiupperbound}
|q_i ( y) | \le C( f_i, \nu ) \frac{N F}{a^{N-1}} e^{ (N-1) B \kappa (y^i ) } e^{ - N f_0 \kappa (y^i) } \frac{1}{|\tilde b ( y^i )| } \| r \|^{N-1}_\infty ,
\end{equation}
where 
$$ f_0 = \min_i \inf_v f_i ( v) $$
and where $ B$ is the explosion rate of the inverse flow.

We wish to establish upper bounds for the partial derivatives of $ q_i $ with respect to $y^j , j \neq i, $ and with respect to $y^i .$ We start with the following control. For any $ j \neq i, $ 
\begin{equation}\label{eq:alsogood}
 | \frac{\partial \beta_{ i \to j , t }^{-1} ( y^j)}{\partial t } | \le \frac{|\tb ( y^j) |}{a} e^{B t }.
\end{equation}
This implies, for all $ y \in {\mathcal A},$ 
\begin{equation}\label{eq:alsogood2}
|\frac{\partial \beta^{-1}_{i \to \cdot } (y  )}{\partial y^i }| \le C( B, a ) e^{ B \kappa ( y^i ) } \frac{1}{|\tb ( y^i ) |}.
\end{equation}
The above considerations lead to the following control. 
\begin{equation}\label{eq:qistrich}
| \partial_\alpha q_i ( y) | \le C( F, B, N, a, A) e^{[(N-1)+ \| \alpha\| ]  B \kappa (y^i ) } e^{ - N f_0 \kappa (y^i ) } \frac{1}{ | \tb ( y^i ) |^{ 1 + \| \alpha \|_{(i)} }}  \sup_{\beta : \| \beta \| \le \| \alpha \|} \| \partial_\beta \prod_{j \neq i } r( x^j )   \|_\infty .
\end{equation} We are now going to check whether $ q_i$ satisfies also points (2) and (3) of the definition of ${\mathcal P}^\infty .$ 

{\bf Case $1:$ Integration with respect to $ y^j $ for $ j \neq i .$} 
By the product form of $ q_i / \bar f , $ for any $j \neq i,$  
\begin{multline*}
( \check  q_i )_j (\check y^j ) = C( f_i , \nu )  e^{ -  \sum_{k \neq j } \int_0^{\kappa ( y^i )}  f_k ( \tilde \gamma_{s, \kappa (y^i ) }^{ -1} (y^k) ) ds}  \prod_{ k \neq i , j } [  \lambda_{i \to k} ( \kappa ( y^i ) , y^k)  r (\beta^{-1}_{i \to k, \kappa (y^i) } ( y^k ) ) ] \\
 \frac{1  }{|\tb ( y^i )| } 1_{ \tilde \gamma^+(0)  } (y^i ) 
\left[ \int  r   ( \beta^{-1}_{i \to j, \kappa (y^i) } ( y^j ) )  e^{ - \int_0^{\kappa (y^i ) } f_j  ( \tilde \gamma_{s, \kappa (y^i ) }^{-1} (y^j ) )ds }   \lambda_{i \to j} ( \kappa ( y^i ) , y^j)  f_j (y^j)   dy^j \right] .
\end{multline*}
But using the inverse transformation of variables $ x^j = \beta_{i \to j , \kappa ( y^i) }^{-1} (y^j ), d x^j = \lambda_{i \to j } ( \kappa (y^i ) ,  y^j ) dy^j  ,$ the last integral equals
$$\int  r(x^j) e^{- \int_0^{\kappa (y^i ) } f_j ( \tilde \gamma_s (x^j ) ) ds } f_j ( \tilde \gamma_{\kappa (y^1 ) } ( \Delta_i^j ( x^j ) )) dx^j  < \infty 
$$
since $f_j$ is bounded and $ r $ a density. 

The same argument shows that $ \partial_\alpha ( \check q_i)_j $ is integrable with respect to any iteration of $ dy^ k$ for any $ k \neq i, j  .$ As a consequence, points (2) and (3) are satisfied provided we do not integrate with respect to $ y^i.$  

{\bf Case $2: $ Integration with respect to $y^i.$} The real difficulty is -- of course --  the integration with respect to $ y^i .$ In order to check for instance that $ ( \check q_i)_i $ is smooth, we have to be able to integrate $ \partial_ \alpha q_i , $ for $ \| \alpha\|_{(i) } = 0 , $ with respect to $ dy^i .$ Using the upper bound \eqref{eq:qistrich},  this means that we have to be able to control
\begin{equation}\label{eq:tobeconsidered}
 \int_0^{\tilde \gamma_\infty ( 0)  } e^{[(N-1) + \| \alpha\| ] B \kappa (y^i ) } e^{ - N f_0 \kappa (y^i ) } \frac{1}{  \tb ( y^i ) } dy^i ,
\end{equation}
or, using the inverse transformation $ y^i = \tilde \gamma_t (0), $ 
$$ \int_0^\infty e^{[(N-1) + \| \alpha\| ] B t } e^{ - N f_0 t }dt . $$

The next example illustrates the above discussion and the problems arising in doing such an integration.

\begin{ex}
Let $ \tb (x) = - ( x- v^* ) $ and $ a_i^j (v) = a > 0 $ for all $ i \neq j .$ Then $\tilde \gamma_t (x) = e^{-t} x + ( 1 - e^{-t}) v^* ,$ and $ \kappa ( y) =  \log (v^* / ( v^* - y)   ) .$ Moreover, 
$$ \beta_{i \to j , \kappa ( y^i )}^{-1} ( y^j ) = v^* \frac{  y^j - y^i }{v^* - y^i }  - a , \; \lambda_{i \to j} (\kappa ( y^i ),  y^j ) = \frac{v^*}{v^* - y^i }.$$ 
Suppose finally that $ f_i ( v) \equiv {\rm f} > 0 ,$ where ${\rm f}$ is a positive constant. Then 
\begin{equation}
 q_i ( y) = C (f_i, \nu ) \frac{N {\rm f}^3}{(v^* )^{N{\rm f} - N +1}}  \left[ \prod_{j\neq i } r \left( v^* \frac{y^j - y^i}{v^* - y^i } - a\right) \right]
(v^* - y^i )^{ N {\rm f} - N } 1_{ 0 \le y^i < v^*}.
\end{equation}
Obviously, any derivative with respect to $ y^j , j \neq i, $ of the above term gives an extra term $ (v^* - y^i )^{-1} .$ 
This shows two things. Firstly, if $ N {\rm f} - N - k > - 1 , $ i.e.\
$$ {\rm f} > 1 + ( k-1)/N,$$
then we may derive $k $ times $q_i ( y) $ with respect to any of the $y^j,$  $j \neq i, $ and still get something integrable in $ y^i $ as $ y^i \uparrow v^*.$ 

Secondly, arguing like this is even too pessimistic, since the presence of the term 
$$  \prod_{j\neq i } r \left( v^* \frac{y^j - y^i}{v^* - y^i } - a\right),$$ 
for $ \check y^i  $ fixed, imposes that $  v^* \frac{y^k - y^i}{v^* - y^i } - a \in supp (r)  , $ for all $k \neq i , $  implying that $ |y^i -  v^*| > \eta $ for some $\eta > 0.$ It is however complicated to iterate this argument, since even if $\nu, $ hence $r,$ is of compact support,  $q_i$ will not be of compact support any more.
\end{ex}

The above discussion leads to the following theorem. 

\begin{theo}\label{theo:10}
We impose Assumptions \ref{ass:1}--\ref{ass:3}, Assumption \ref{ass:5} and Assumption \ref{ass:9}. Let $ f_0 = \min_i \inf_v f_i ( v) .$  

If there exists $k^* $ such that $f_i \in H(k^*, F) ,$ for all $ 1 \le i \le N,$ and $\tb \in H(k^* , B) $ and 
\begin{equation}\label{eq:sufficient}
B k^* < N f_0  - (N-1) B,
\end{equation}
then the invariant measure $m$ of the process $(X_t)_{ t \geq 0} $ possesses a density $ p \in C^{k^*} ( {\mathcal A} ) ,$ where 
 $ {\mathcal A} = \{ x \in \R_+^N :  x^i \notin {\mathcal E}  , \; \; \forall 1 \le i \le N  \} .$ 
\end{theo}

As a consequence, if there is some balance of the explosion rate $ e^{Bt } $ of the inverse flow and the survival rate $ e^{ - f_0 t } $ of the system, then the invariant density is regular up to some order which is precisely given by this balance. Hence we can exhibit at least one regime in which we are able to say something about (some) regularity of the invariant measure. Of course, the conditions given in the theorem are far from being sharp and it would be interesting to find other regimes where regularity of the invariant density can be shown.

\begin{proof}
Using the transformation of variables $ t = \kappa (s) , $ one sees that equation \eqref{eq:sufficient} is equivalent to 
$$\int_0^\infty e^{[(N-1) B  - N f_0 + k^*B  ] t}  dt =  \int^{\tilde \gamma_\infty ( 0) }_0 e^{[(N-1) B  - N f_0 + k^*B  ] \kappa (s) } \frac{1}{|\tb (s)| } ds < \infty ,$$
which shows that the expression arising in \eqref{eq:tobeconsidered} is finite. 

Following Step 3.\ of the above discussion, we introduce $ \mathcal P^{(k ^* )} = \{ p ( x) dx , p : \R^N \to \R_+ , \int p (x) dx = 1 \}$ satisfying the following three points.
\begin{enumerate}
\item
$ p \in C^{k^*} ( {\mathcal A}) .$
\item
$ \check p_j (\check x^j ): = \int p ( x) (f_j ( x^j )/\bar f (x) )  dx^j $ belongs to $ C^{k^*}  (\{ x \in \R^{N-1} : x^i \notin {\mathcal E} , \forall i \}),$ for all $j .$ 
\item
For all  $ \ell \le  N-1, $ for all $ k_1< \ldots < k_\ell  \in \{ 1, \ldots , N \} \setminus \{j \} ,$  for all $ \alpha $ such that $ \| \alpha \|_{(k_1)} =\ldots =\| \alpha \|_{(k_\ell )}=  0$ and $ \| \alpha \| \le k^*,$ $ \int_{\R^\ell}  | \partial_\alpha  \check p_j ( \check x^j ) |d x^{k_1} \ldots d x^{k_\ell }  < \infty .$
\end{enumerate} 

Under the conditions of the theorem, in particular \eqref{eq:sufficient}, the Step 3.\ of the above discussion shows : If $ Z_n \sim p ( x) dx  \in {\mathcal P}^{(k^*)}  $  then also ${\mathcal L}( Z_{n+1})   \in {\mathcal P}^{(k^*)} .$ By Nummelin splitting and since the regeneration law $ \nu (x) dx $ belongs to $ {\mathcal P}^{(k^*)}, $ this implies the assertion, following the lines of Step 2.\ of the proof of Theorem \ref{theo:8}.
\end{proof}

\begin{cor}
For the system of interacting neurons introduced in Example \ref{ex:neuro}, if $f_0 > \lambda, $ the invariant density is at least $k-$times differentiable on ${\mathcal A} ,$ for any $ k < N f_0 / \lambda - (N-1) .$ 
\end{cor}

\begin{proof}
We have already shown that $ ( N, \tt , \ii) $ is good for $ \ii = ( 1, 2, \ldots , N ),$ for any $ \tt .$ Thus, the conditions of Theorem \ref{theo:7} are fulfilled. Moreover, $ Y_t (x) =e^{ - \lambda t } Id  $ for all $ x \in \R^N, $  with $ Id $ the $N \times N -$identity matrix, and $ A^i (x) = diag ( 1, \ldots , 1 , 0 , 1, \ldots, 1 ) $ is the diagonal matrix consisting of all $1's,$ except for the $i-$th entry which is a $0.$ As a consequence, 
$$ Y_t ( \Delta_i ( x) ) A^i (x) = e^{  - \lambda t } \left( 
\begin{array}{cccccccccccc}
e_1 & |  & \cdots &|&  e_{i-1} &| &0 &|& e_{i+1 } & \cdots &|& e_N 
\end{array} 
\right) ,$$
where $e_k$ are the unit vectors in $\R^N.$ Moreover, 
$$ b ( \gamma_t ( \Delta^i (x) ) = \left(
\begin{array}{c}
* \\
\vdots \\
* \\
\lambda e^{  - \lambda t } m \\
* \\
\vdots \\
* 
\end{array} \right) 
\leftarrow \mbox{ $i-$th coordinate},$$
implying that the columns of  $ Y_t ( \Delta_i ( x) ) A^i (x) $ and $   b ( \gamma_t ( \Delta^i (x) ) )$ span $ \R^N $ for all $ t .$ The assertion then follows from Theorem \ref{theo:10}, observing that $ B = \lambda .$ 
\end{proof}

\section*{Appendix}
\subsection{Proof of Theorem \ref{theo:invmeasure}.} 
W.l.o.g.\ we take $i=1,$ i.e.\ we study the smoothness of the invariant density of the first particle. 

We rely on the following smoothness criterion in dimension $1$ that we quote from Bally and Caramellino (2011) \cite{Bally-Caramellino}. Let $\lambda $ be the Lebesgue measure on $ \R .$ Recall that we work locally on $S_{d , k +2 } = \{ v \in \R : (k+2) A < |v|, |\tb (v) | > d \} , $ for $ d > (k+2) AB.$ We fix some $ m \geq 1 .$ Let
$W^{m, p }_{\lambda} ( S_{d , k +2 }) $ be the space of all functions $ \phi \in L^p ( \lambda ) $ such that for all $ \ell \le m , $ for all $g \in C^\infty_c ( S_{d , k +2 }) ,$ 
\begin{equation}\label{eq:Sobolev}
 \int \partial_\ell g (x) \phi  (x) \lambda ( dx) = (- 1)^\ell \int g(x) \theta_\ell ( x) \lambda ( dx) , 
\end{equation} 
for some functions $ \theta_\ell \in L^p ( \lambda), $ for all $ 1 \le \ell \le m .$ If \eqref{eq:Sobolev} holds, then we write $ \partial^{\lambda}_\ell \phi := \theta_\ell .$ If $\phi $ and all $\theta_\ell $ are bounded on $S_{d , k +2 }, $ then we say that $ \phi \in W^{m, \infty  }_{\lambda} ( S_{d , k +2 })$ and we introduce 
$$ \| \phi\|_{W^{m, \infty  }_{\lambda} ( S_{d , k +2}) } :=  \sup_{ 0 \le \ell \le m }  \sup_{ x \in S_{d , k +2 } } |\partial^{\lambda}_\ell \phi (x)   | .$$ We quote the following theorem from \cite{Bally-Caramellino}.

\begin{theo}[Theorem 8 of \cite{Bally-Caramellino}]\label{theo:vlad} 
Write $ \pi (dx) = \phi (x) \lambda ( dx) .$ If $\phi \in W^{m, \infty  }_{\lambda} ( S_{d , k+2 }) , $ then $ \pi (dx) = \pi ( x) dx $ locally on $S_{d , k +2 } .$ Moreover, $\pi \in C^{m-1} (S_{d , k +2 }) $ and for all $ \ell \le m-1,$ 
$$ \sup_{ x \in S_{d , k +2 }} |   \pi^{(\ell ) }   ( x) | \le C \| \phi\|_{W^{m, \infty  }_{\lambda} ( S_{d , k+2 }) }  ,$$
where the constant $C$ does not depend on $\phi .$ 
Finally, $ \pi^{(m-1) } $ is H\"older- continuous of order $\tilde \alpha $ for any $ \tilde \alpha < 1 , $ and 
$$ \sup_{ x , x' \in S_{d , k +2 }} |   \pi^{(m-1 ) }  ( x) - \pi^{(m-1)} ( x')  | \le C( \tilde \alpha )    \| \phi\|_{W^{m, \infty  }_{\lambda} ( S_{d , k +2 }) } |x-x'|^{\tilde \alpha },$$
where the constant $ C( \tilde \alpha ) $ does not depend on $\phi.$  
\end{theo}

We now show how to apply the above theorem. We work under stationary regime and suppose that $X_0 \sim m.$ In this case, $ Z_k \sim \frac{1}{m ( \overline f)} m  ( \bar f \cdot ) $ for any $k, $ where we recall that $ Z_k = X_{T_k - } .$ 

Now, let $g  \in C^\infty_c (S_{d , k +2 }) $ be a smooth test function having compact support in $ S_{d , k +2 }.$ We start with a  control on $ \pi ( g) .$ Using \eqref{eq:useful}, we obtain
\begin{equation}\label{eq:muz}
\pi ( g) =\sum_{i=1}^N \int_{\R^N} m ( dx)  f_i ( x) \int_0^\infty  e( \Delta_i ( x) , t)   g( \tilde \gamma_t (\Delta^1_i ( x)  ) ) dt . 
\end{equation}

{\bf Step 1.}
We work with a fixed value of $y:= \Delta_i ( x) $ and wish to use the change of variables 
$$ s = \tilde \gamma_t ( y^1) , ds =  \tb (s) dt  .$$
By the support properties of the function $g,$ $ g ( \tilde \gamma_t (y^1)) \neq 0 $ implies that $ \tilde \gamma_t ( y^1 ) \notin  {\mathcal E}.$ Therefore, $ y^1 = \Delta^1_i ( x^1) \notin {\mathcal E} $ neither and $ t \mapsto \tilde \gamma_t ( y^1) $ invertible for all $t \in [0, \infty [.$ Let $ \kappa_{y^1} ( s) $ be the associated inverse function. Then we may rewrite 
$$ \int_{0}^{\infty}  e( y , t)   g( \tilde \gamma_t (y^1 ) ) dt = \int_{ y^1 }^{\tilde \gamma_\infty (y^1) }  e (y, \kappa_{y^1}   ( s) )  \frac{g(s) }{\tb ( s) } ds $$
and obtain from \eqref{eq:muz} that
$$ d \pi = \phi d \lambda ,$$
$\lambda $ the Lebesgue measure on $\R,$ where  
\begin{equation}\label{eq:nudens}
\phi ( s) =  \sum_{i=1}^N \int m ( dx) f_i ( x)  1_{ \tilde \gamma^+  (\Delta_i (x) ) } (s)    e (\Delta_i (x) , \kappa_{\Delta_i^1 (x) }    ( s) ) \frac{1}{|\tb ( s)| }  .
\end{equation}
Notice that $\phi$ is bounded on $S_{d , k +2} , $ with bound given by 
$$ \sup_{ s \in S_{d , k +2} } \phi ( s) \le \frac{NF}{ d },$$
since $ | \tb ( s) | \geq d .$ 
Therefore, we are exactly in the situation of Theorem \ref{theo:vlad}.  In particular, $\pi$ possesses a bounded Lebesgue density $\pi $ on $ S_{d, k +2 } $ which is precisely given by $ \pi ( s) = \phi ( s)$ for all
$ s \in S_{d , k +2 }.$ 

{\bf Step 2.} Let us come back to \eqref{eq:muz}. In general, we will have to consider expressions of the form
\begin{equation}
\pi_H ( g) := \sum_{i=1}^N \int m(dx) f_i ( x) \int_0^\infty e( \Delta_i ( x), t ) H( \gamma_t ( \Delta_i ( x)) ) g ( \tilde \gamma_t (\Delta^1_i ( x))) dt ,
\end{equation}
where $ \gamma_t (y) = ( \tilde \gamma_t (y^1 ) , \ldots , \tilde \gamma_t (y^N) ) $ is the joint flow of the $N$ particles and where $ H : \R^N \to \R $ is bounded and smooth, for a smooth function $ g : \R \to \R$ such that $ g  \equiv 0 $ in a neighborhood of $0$ and such that $ g (v) \neq 0 $ implies that $\tb ( v) \neq 0 ,$ i.e.\ $ v \notin {\mathcal E}.$ Notice that for $ H \equiv 1,$ 
$$ \pi_1 (g) = E_m ( g(X_t^1 ) ) = \pi ( g)  \mbox{ and therefore } \pi_1 (1) = 1.$$ 
In order to clarify the structure of the problem, let us consider 
$$
\pi_H ( g'') = \sum_{i=1}^N \int m ( dx) f_i ( x)   \int_0^\infty \frac{e ( \Delta_i ( x) , t)}{\tb ( \gamma_t (\Delta^1_i ( x) ) )  }  H( \gamma_t ( \Delta_i ( x)) )   g'' ( \tilde \gamma_t ( \Delta^1_i (x)  )) \frac{d \tilde \gamma_t ( \Delta^1_i ( x)  )}{dt }  dt  .
$$
Integration by parts yields  
\begin{multline*}
\int_0^\infty \frac{e (y , t)}{\tb ( \tilde \gamma_t (y^1 ) } H( \gamma_t ( y) )   g'' ( \tilde \gamma_t ( y^1 )) \frac{d \tilde \gamma_t ( y^1 )}{dt }  dt  
= [ \frac{e (y,t)}{\tb( \tilde \gamma_t (y^1) ) } H( \gamma_t (y) )   g' ( \tilde \gamma_t (y^1) ) ]_0^\infty \\
- \int_0^\infty \frac{d}{dt} \left[ \frac{e (y , t)}{\tb ( \tilde \gamma_t (y^1)  )  }H( \gamma_t (y) ) \right]    g' ( \tilde \gamma_t (y^1) ) dt .
\end{multline*}

The boundary terms that arise in this integration by parts formula have to be studied carefully. Firstly, exploiting Assumption \ref{ass:2},
$$ \left( \frac{e (y ,t)}{\tb ( \tilde \gamma_t (y^1)  ) } H( \gamma_t ( y) )   g' ( \tilde \gamma_t (y^1)) \right)_{| t = \infty } = 0.$$

Moreover,
$$  \left( \frac{e (y ,t)}{\tb ( \tilde \gamma_t (y^1)  ) } H( \gamma_t ( y) )   g' ( \tilde \gamma_t (y^1)) \right)_{| t = 0 }  =  \frac{g'(y^1)  }{\tb ( y^1 ) } H( y) .$$
Finally, we calculate
\begin{eqnarray*}
 \frac{d}{dt} \left[\frac{e (y , t)}{\tb ( \gamma^1_t (y) ) } H( \gamma_t (y) ) \right] &=& - e(y , t)[ \frac{ [\bar f   H] ( \gamma_t (y) )  - <\nabla H  , b  >( \gamma_t (y) ) + H( \gamma_t(y) ) \tb' ( \tilde \gamma_t (y^1) ) }{ \tb ( \tilde \gamma_t (y^1))   }  ]\\
&= : &- e (y, t) G( H) (\gamma_t( y)) ,
\end{eqnarray*} 
where 
\begin{equation}\label{eq:GH}
G(H) (x) =  \frac{ [\bar f   H] ( x )  - <\nabla H  , b  >( x ) + H( x ) \tb' ( x^1 ) }{ \tb ( x^1)   }  .
\end{equation}

Hence
\begin{equation}\label{eq:important}
\pi_H ( g'') = - \sum_{i=1}^N \int m (dx) f_i ( x)  \frac{g'(\Delta^1_i (x^1)) }{\tb ( \Delta_i^1(x^1)  ) }  H( \Delta_i ( x) ) 
+ \pi_{G(H) } ( g') . 
\end{equation}

Let us study the first term in the above expression,  
$$ T_1 :=  -  \sum_{i=1}^N \int m  (dx) f_i( x )  \frac{g'(\Delta_i^1 ( x^1)) }{\tb ( \Delta_i^1 (x^1)) } H( \Delta_i ( x) ) ,$$
where $ \Delta_i^1 ( x^1) = 0 $ for $ i = 1.$   
Since $g'(0) = 0$ ($g \equiv 0$ in a neighborhood of $0$),  the term $i=1$ does not appear in the above sum. Therefore, the above term equals
\begin{equation}
T_1  = - \sum_{i=2}^N \int m (dx) f_i ( x)   \frac{[g \circ\Delta_i^1]' (x^1)   }{\tb' ( \Delta_i^1  (x^1)  ) (\Delta_i^1 )' (x^1 )  } H( \Delta_i ( x) )  ,
\end{equation}
where $ (\Delta_i^1)' (x^1 ) \neq 0 $ by Assumption \ref{ass:5}.

We apply \eqref{eq:useful} to the test function $  f_i ( x)   \frac{[g \circ \Delta_i^1]' (x^1)   }{\tb ( \Delta_i^1  (x^1)  )(\Delta_i^1 )' (x^1 )  } H( \Delta_i ( x) ) := \sum_{i=2}^N G_i (H) (x) [g \circ \Delta_i^1]'   (x^1)   ,$ where 
\begin{equation}\label{eq:GHi}
G_i ( H) (x) :=     \frac{f_i ( x) }{\tb ( \Delta_i^1  (x^1)  )(\Delta_i^1 )' (x^1 )  } H( \Delta_i ( x) )  ,
\end{equation}
for any $ i \geq 2.$ As a consequence, 
\begin{eqnarray}\label{eq:221}
T_1 &=&
-\sum_{j=1}^N  \sum_{i=2}^N \int m ( dx) f_j (x) \int_0^\infty e(\Delta_j ( x) ,t) \nonumber \\
&& \quad \quad \quad   
 \frac{[g  \circ  \Delta_i^1]'  \circ \gamma^1_t ( \Delta_j (x) )  )}{  \tb (  \Delta_i^1 \circ \tilde \gamma_t ( \Delta^1_j (x) )  ) (\Delta_i^1 )' (\tilde \gamma_t ( \Delta^1_j (x) )) } f_i (\gamma_t ( \Delta_j(x) ) ) H( \Delta_i ( \gamma_t ( \Delta_j(x) ) )) dt  \nonumber \\
&=&- \sum_{i=2}^N  \pi_{G_i ( H) } ([ g \circ \Delta_i^1]' ) .
\end{eqnarray}

Resuming the above discussion, we obtain 
\begin{equation}\label{eq:IPP}
\pi_H ( g'') = \pi_{G(H) } ( g') - \sum_{i=2}^N \pi_{G_i ( H) } ( [g \circ \Delta_i^1]' ) ,
\end{equation}
for any smooth test function $g: \R \to \R $ such that $ supp ( g) \subset {\mathcal E}^c$ and such that $ g \equiv  0$ on a neighborhood of $0,$ 
where $G(H) $ and $G_i ( H) $ are given in \eqref{eq:GH} and \eqref{eq:GHi}. 

Of course, we want to iterate the above procedure. For that sake, we have to be sure that $ g \circ \Delta_i^1 $ appearing in the second term of \eqref{eq:IPP} belongs still to the class of functions which are admissible in order to obtain \eqref{eq:IPP}, i.e.\ $ g \circ \Delta_i^1 \equiv 0$ in a neighborhood of $0$ and $ |\tb (v) | $ lower bounded on the support of $ g \circ \Delta_i^1 .$ This is why we have to restrict attention to the set $S_{d, k +2}.$ (We will give more details in the next step.) 
%

{\bf Step 3.} 
For a function $ H$ defined on $\R^N  , r \in \N $ and any open set $B \subset \R^N,$  we denote 
$$ \| H\|_{B, r, \infty } := \sum_{k=0}^r \sum_{ |\alpha | = k } \sup_{ x \in B} | \partial^\alpha H  (x) |  .$$
Write $ \bar S_{d, k +2 } : = \{ x \in \R^N : x^1 \in S_{d, k +2 } \} = S_{d, k +2 } \times \R^{N-1} .$ 
Then
\begin{equation}
\| G (H) \|_{\bar S_{d, k +2 } , 0, \infty } \le \frac{NFB}{d} \| H\|_{\bar S_{d, k +2} , 1, \infty } 
\end{equation}
and
\begin{equation}\label{eq:controlGHi}
\| G_i ( H) \|_{\bar S_{d, k +2 } , 0, \infty } \le \frac{F}{a d} \| H \|_{ \bar S_{d, k +2} , 0, \infty } .
\end{equation}
Finally, the second term in \eqref{eq:IPP} contains the test function $g$ transported by the jump term, i.e. the function $ g \circ \Delta_i^1.$ Its support is contained in 
\begin{equation}\label{eq:sautsupp}
(\Delta_i^1 )^{-1}  ( S_{d, k +2 })  \subset S_{d - AB, k+1}.
\end{equation}
This last inclusion can be seen as follows. $ \Delta_i^1 (x^1 )  \in S_{d, k +2} $ implies on the one hand that
$|\tb ( \Delta_i^1  (x^1)  )|  = |\tb ( x^1 + a_i^1 ( x^1 ) )| \geq d .$ But
$$ \tb (x^1 ) = \tb ( x^1 + a_i^1 (x^1) ) - \tb' ( x^1 + \vartheta a_i^1 ( x) ) a_i^1  (x) ,$$
for some $ \vartheta \in ]0, 1 [,$ whence 
\begin{equation}\label{eq:tblb}
 | \tb ( x^1 ) | \geq d -AB . 
\end{equation}
On the other hand, $ |x^1 | \geq | \Delta_i^1 ( x^i ) | - A \geq (k+2) A - A = (k+1) A.$ As a consequence, $ x^1 \in S_{d- AB, k+ 1}.$ 

Therefore, coming back to \eqref{eq:IPP}, we have 
$$ \| \pi_H ( g'') \|_\infty \le \left[ \frac{NFB}{d} \|H\|_{\bar S_{d, k +2 } , 1, \infty } + \frac{F}{a(d- AB)} \| H\|_{\bar S_{d- AB, k+1} , 0, \infty } \right] \| g '\|_\infty $$
(recall that $ d > (k+2)   AB$).

{\bf Step 4.} As a consequence, we may iterating \eqref{eq:important} $k+1$ times. For $g \in C^\infty_c ( S_{d, k +2 }),$ we obtain
\begin{multline}\label{eq:oufff}
\pi ( g^{(k+1)} ) = \pi_1 ( g^{(k+1) } ) = 
\pi_{G(1) } (g^{(k)}) - \sum_{i_1= 2}^N \pi_{G_{i_1} ( 1) } ( ( g \circ \Delta_{i_1}^1)^{(k)} ) \\
= \pi_{G^2 (1) } ( g^{(k-1) } ) - \sum_{i_1= 2}^N \pi_{G \circ G_{i_1} ( 1) } ( ( g \circ \Delta_{i_1}^1)^{(k-1)} )\\
+\sum_{i_2=2}^N  \sum_{i_1= 2 }^N\pi_{G_{i_2 } \circ G_{i_1} (1)} (( g \circ \Delta_{i_1}^1 \circ \Delta_{i_2}^1 
)^{(k-1)} ) . 
\end{multline}
These iterations give rise to the following tree.  Let 
\begin{equation}\label{eq:tree}
{\mathcal V} := \{ (i_1, \ldots , i_{k+1} ) : i_l \in \{2, \ldots , N \} \cup \{ 0 \} \} ,
\end{equation}
where each choice $ i_l = 0 $ will be interpreted as choice of $G.$ Of course, each choice $ i_l   \in \{ 2, \ldots , N\} $ corresponds to the choice of $G_{i_l}.$ We write ${\mathbf v}$ for all elements of ${\mathcal V} $ and put 
$$ G_{\mathbf v } := G_{i_{k+1}} \circ G_{i_k} \circ \ldots \circ G_{i_1} , $$
for $ \mathbf v = (i_1, \ldots , i_{k+1} ),$  where $G_0 := G ,$ and 
$$ \Delta_{\mathbf v} := \Delta_{i_1} \circ \ldots \circ \Delta_{i_{k+1} } , $$
with $ \Delta_0 \equiv id.$ Finally we put $ sgn ( \mathbf v) := (-1)^{\sum_{l=1}^{k+1} 1_{ \{i_l  \neq 0 \} }} .$ 

Then 
\begin{equation}\label{eq:IPPbis}
\pi ( g^{(k+1)} )  = \sum_{ \mathbf v \in {\mathcal V} } sgn ( \mathbf v) \pi_{ G_{\mathbf v } ( 1) } ( g \circ \Delta^1_{\mathbf v} ) .
\end{equation}

In order to conclude the proof and to apply Theorem \ref{theo:vlad} with $ m=k+1,$ we have to show that each term appearing in the last expression of \eqref{eq:IPPbis} can be written as an integral with respect to the Lebesgue measure. This can be shown by using a simple change of variables. Consider e.g.\ the expression $\pi_{G_{i_{k+1}} \circ \ldots \circ G_{i_2 } \circ G_{i_1} (1)} ( g \circ \Delta_{i_1}^1 \circ \Delta_{i_2}^1 \circ \ldots \circ \Delta_{i_{k+1}}^1  ) ,$ corresponding to  $ \mathbf v = ( i_1, \ldots, i_{k+1} ) \in \{2, \ldots , N\}^{k+1}  $ (no choice of $ 0$).  

Firstly, 
$$supp (  \Delta_{\mathbf v}^1)^{-1} ( S_{d, k+2} )  \subset S_{d - (k+1) AB ,  1} , $$
which follows from \eqref{eq:sautsupp}.
Moreover, by \eqref{eq:controlGHi}, 
$$ \| G_{\mathbf v } ( 1)  \|_{ \bar S_{d- (k+1) AB, 1} , 0, \infty } \le C( F, B, k , a) \frac{1}{ d (d-AB) \ldots (d- (k+1) AB) } . $$
Recall that
$$
\pi_{G_{\mathbf v } (1)} ( g \circ \Delta_{\mathbf v }^1) =  
\sum_l \int m (dx) f_l ( x) \int_0^\infty e( \Delta_l ( x), t ) (G_{\mathbf v } (1) ( \gamma_t (\Delta_l (x) )))  g \circ \Delta_{\mathbf v}^1 ( \tilde \gamma_t ( \Delta_l^1 ( x^1 ) ) ) dt .
$$
Therefore, we use the change of variables $ s = \Delta_{\mathbf v}^1  (\tilde \gamma_t (\Delta_l^1 (x^1) )),  $
$$ ds = [\Delta_{\mathbf v}^1]' (  \tilde \gamma_t (\Delta_l^1 (x^1) )) \tb (\tilde \gamma_t (\Delta_l^1 (x^1) )) dt  ,$$
with 
$$ | [\Delta_{\mathbf v}^1]' (  \tilde \gamma_t (\Delta_l^1 (x^1) )) | > C ( k, a  ) $$
(recall Assumption \ref{ass:5}) 
and 
$$ | \tb (\tilde \gamma_t (\Delta_l^1 (x^1) ))| \geq d - (k+1) AB > AB ,$$
since $ \tilde \gamma_t (\Delta_l^1 (x^1) ) \in S_{d - (k+1) AB ,  1} $ by the support property of $g \circ \Delta_{\mathbf v}^1 .$
We write $ \kappa_{\mathbf v, l } (s, x^1) $ for the inverse function of $ t \mapsto \Delta_{\mathbf v}^1  (\tilde \gamma_t (\Delta_l^1 (x^1) )) $ and obtain that  
$$
\pi_{G_{\mathbf v }(1)} ( g \circ \Delta_{\mathbf v}^1  ) = 
\int  g (s ) \theta_{\mathbf v} (s) dy ,
$$
where
\begin{multline*}
\theta_{\mathbf v} (s) = \sum_l \int m (dx) f_l ( x) 1_{ \Delta_{\mathbf v } ( \tilde \gamma^+ ( \Delta_l^1 (x^1 ) ) ) } (s) \\
e( \Delta_l (x) , \kappa_{ \mathbf v, l } (s, x^1)) G_{\mathbf v} (1) ( \gamma_{\kappa_{ \mathbf v, l } (s, x^1)} ( \Delta_l (x) )) \\
\left([\Delta_{\mathbf v}^1]' (  \tilde \gamma_{  \kappa_{\mathbf v, l } (s, x^1)   } (\Delta_l^1 (x^1) )) \tb (\tilde \gamma_{  \kappa_{\mathbf v , l } (s, x^1) }  (\Delta_l^1 (x^1) ))  \right)^{-1} .
\end{multline*}
Moreover, we have the bound 
$$\| \theta_{\mathbf v} \|_{ S_{d  ,  k+2}  , 0, \infty } \le C( N, F, B, k, a)  \frac{1}{ d (d-AB) \ldots (d- (k+1) AB) } \frac{1}{AB} . $$

Putting all things together, we see from the above considerations that
$$ \pi ( g^{(k+1)} ) = \int g(s) \Theta  ( s) \lambda ( ds) ,$$
for some Lebesgue density $\Theta ,$ where 
$ \sup_{s \in S_{d, k+2} } | \Theta ( s) |  \le  C( k, F, B, a, A,  N, d)  $ for some constant depending only on the classes of functions $H( k , F),  $ $H(k+1, B)$ and on $a, A $ and $d.$ 

{\bf Step 5.} By Theorem \ref{theo:vlad}, applied with $ m=k+1,$ we deduce that $\pi \in C^k ( S_{d , k+2 }), $ and $ \pi^{(k) } $ is H\"older continuous of any order $ \tilde \alpha < 1.$ Moreover, 
$$ \sup_{\ell \le k , w \in S_{d , k+2 } } |  \pi^{(\ell )} ( w)  | + \sup_{w \neq w' , w, w' \in S_{d , k +2 } } \frac{\pi^{(k)} (w) - \pi^{(k)} (w') }{|w-w'|^\alpha } \le C ,$$
for some constant which does only depend on the classes of functions $H( k , F),  $ $H(k+1, B)$ but not on $f_i$ or $\tb.$ 
This finishes our proof. $\hfill \qed $

\subsection{On the equilibria of one-dimensional dynamical systems}\label{sec:flow}
For the one-dimensional flow $\tilde \gamma_t $ of Section \ref{sec:23}, we have that 
\begin{equation}
\tilde \gamma^+ ( {\mathcal E}^c) \subset {\mathcal E}^c .
\end{equation}

\begin{proof}
Let $ v $ be such that $ \tb ( v) \neq 0 $ and suppose that there exists $0 < t^* < \infty $ such that $ \tilde \gamma_{t^* }   (v) = v^* $ with $ \tb ( v^*) = 0.$ Using Taylor's formula in a neighborhood of $v^*, $ we obtain $ \tb ( x) = \tb ' (\xi ) ( x - v^* ) , $ for some $ \xi \in ]x, v^* [ \cup ]v^* , x[ .$ Since $ v^* $ is necessarily attracting, $ \tb ' ( \xi ) \le 0, $ for $ |x- v^* | \le \varepsilon, $ for some appropriate $ \varepsilon > 0 .$ Therefore, $ \tb ( x) \le B (v^* - x) ,$ if $ x \in ]v^* - \varepsilon, v^* [ , $ and $ \tb ( x) \geq B ( v^* - x ) , $ for $x \in ]v^*, v^* + \varepsilon  [ ,$ where $B  $ is an upper bound on $  \| \tb' \|_\infty .$ Suppose w.l.o.g. that $ v < v^* .$ Let $ t_\varepsilon = \inf \{ t : \tilde \gamma_t ( v) \in ]v^* - \varepsilon, v^* [ \} .$ By assumption, $t_\varepsilon < t^* < \infty .$ A simple comparison argument shows that on $ [t_\varepsilon, t^* ], $ 
$$ \tilde \gamma_s ( v) \le \bar \gamma_s (v) $$
where $ d \bar \gamma_s ( v) = - B ( \bar \gamma_s ( v)  - v^* ) ds , s \geq t_\varepsilon , $ with initial condition $ \bar \gamma_{t_\varepsilon } ( v) = \tilde \gamma_{t_\varepsilon} = v^* - \varepsilon.$ But the equation for $ \bar \gamma $ possesses an explicit solution given by 
$$ \bar \gamma_s ( v) = e^{- B ( s - t_\varepsilon ) } [v^* - \varepsilon] + ( 1 - e^{- B ( s - t_\varepsilon ) } ) v^* < v^* $$
for all $ s > t_\varepsilon. $ As a consequence, $\tilde \gamma_{t^* } ( v) \le \bar \gamma_{t^* } ( v) < v^* ,$ which is a contradiction. 
This finishes our proof. 
\end{proof}

\section*{Acknowledgments}
This research has been conducted as part of the project Labex MME-DII (ANR11-LBX-0023-01) and as part of the activities of FAPESP  Research, Innovation and Dissemination Center for Neuromathematics (grant
2013/ 07699-0, S.Paulo Research Foundation).


\begin{thebibliography}{99}
\bibitem{athreyaney} 
{\sc Athreya, K.B.,  Ney, P.}
\newblock A new approach to the limit theory of recurrent Markov chains.
\newblock{ \em Trans. Am. Math. Soc. 245}  (1978), 493-501.

\bibitem{Bally-Caramellino}
{\sc Bally, V., Caramellino, L.}
\newblock Riesz transform and integration by parts formulas for random variables.
\newblock{ \em Stochastic Processes and their Applications, Vol 121, No6} (2011) 1332 -- 1355.

\bibitem{BallyClementine}
{\sc Bally, V., Cl\'ement, E.}
\newblock Integration by parts formula with respect to jump times for stochastic differential equations.  
\newblock {\em Stochastic Analysis, Springer Verlag, } 2010. 

\bibitem{bally-rey}
{\sc Bally, V., Rey, C.}
\newblock Approximation of Markov semigroups in total variation distance.
\newblock Available on https://hal-upec-upem.archives-ouvertes.fr/hal-01110015 , 2015.

\bibitem{micheletal} 
{\sc Bena\"im, M., Le Borgne, S., Malrieu, F., Zitt, P.-A.}
\newblock Qualitative properties of certain piecewise deterministic Markov processes. 
\newblock {\em Annales de l'IHP 51} (2015) 1040--1075.

\bibitem{bgj}
{\sc Bichteler, K., Gravereaux, J.B., Jacod, J.}
\newblock Malliavin calculus for processes with jumps.
\newblock {\em Number 2 in Stochastic Monographs. Gordon and Breach, London, } 1987.


\bibitem{Tyran}
{\sc Biedrzycka, W., Tyran-Kaminska, M}
\newblock Existence of invariant densities for semiflows with jumps
\newblock {\em Journal of Mathematical Analysis and Applications 435} (2016), 61 - 84.

\bibitem{bismut}
{\sc Bismut, J.M.}
\newblock Calcul des variations stochastiques et processus de sauts.
\newblock{ \em Z.W. 63} (1983), 147--235


\bibitem{carlen-pardoux}
{\sc Carlen, E., Pardoux, E.}
\newblock Differential calculus and integration by parts on Poisson space. 
\newblock {\em In: Stochastics, algebra and analysis in classical and quantum dynamics (Marseille, 1988), Math. Appl. 59, Kluwer, Dordrecht} (1990) pp. 63--73.

\bibitem{coquio}
{\sc Coquio, A., Gravereaux, J. B.}
\newblock Calcul de Malliavin et r\'egularit\'e de la densit\'e d'une probabilit\'e invariante d'une cha\i {\i}ne de Markov.
\newblock { \em Annales de l'institut Henri Poincar\'e (B) Probabilit\'es et Statistiques 28} (1992) 431--478.



\bibitem{costa-dufour}
{\sc Costa, O.L.V., Dufour, F.} 
\newblock Stability and ergodicity of piecewise deterministic Markov processes.
\newblock {\em SIAM J. Control Optim. 47} (2008) 1053--1077.


\bibitem{Davis93}
{\sc Davis, M.H.A.}
\newblock Markov models and optimization.
\newblock{ \em Monographs on Statistics and Applied Probability, vol. 49} Chapman $\&$ Hall, London. (1993)

\bibitem{Denis}
{\sc Denis, L.}
\newblock A criterion of density for solutions of Poisson-driven SDEs.
\newblock {\em Probability Theory and Related Fields 118} (2000), 406--426.



\bibitem{bresiliens}
{\sc Duarte, A., Ost, G.}
\newblock A model for neural activity in the absence of external stimuli.
\newblock To appear in Markov Proc. Related Fields 2016, available on http://arxiv.org/abs/1410.6086. 

\bibitem{fournier}
{\sc Fournier, N.}
\newblock Jumping SDEs: absolute continuity using monotonicity. 
\newblock {\em Stochastic Process. Appl. 98} (2002), 317--330.


\bibitem{federer}
{\sc Federer,  H.}
\newblock Geometric measure theory. 
\newblock Reprint of the 1969 edition, Springer, 1996.

\bibitem{antonio-eva}
{\sc Galves, A., L\"ocherbach, E.}
\newblock Modeling networks of spiking neurons as interacting processes with memory of variable length.
\newblock To appear in Journal de la Soci\'et\'e fran\c{c}aise de Statistiques 2016, available on http:// arxiv.org/abs/1502.06446  

\bibitem{pierre}
{\sc Hodara, P., Krell, N., L\"ocherbach, E.}
\newblock Nonparametric estimation of the spiking rate in systems of interacting neurons.
\newblock Work in progress, to be submitted, 2016. 


\bibitem{js}
{\sc Jacod, J., Shiryaev, A.N.}
\newblock Limit theorems for stochastic processes.
\newblock Second edition, Springer-Verlag, Berlin, 2003.

\bibitem{nummelin} 
{\sc Nummelin, E.}
\newblock  A splitting technique for Harris recurrent Markov chains.
\newblock{ \em  Z. Wahrscheinlichkeitstheorie Verw. Geb. 43} (1978), 309--318.

\bibitem{picard}
{\sc Picard, J.}
\newblock On the existence of smooth densities for jump processes
\newblock {\em Probability Theory and Related Fields 105} (1996), 481--511.


\bibitem{poly}
{\sc Poly, G.}
\newblock Absolute continuity of Markov chains ergodic measures by Dirichlet forms methods.
\newblock Available on https://hal.archives-ouvertes.fr/hal-00690738 , 2012.






\end{thebibliography}
\end{document}